\documentclass[a4paper]{article}

 \usepackage{a4wide}
 \usepackage{authblk}

 \usepackage{cite}
 \usepackage{graphicx}
 \usepackage{amsmath,amssymb}
 \usepackage{amsthm}
   \theoremstyle{plain}
   \newtheorem{theorem}{Theorem}
   \newtheorem{proposition}[theorem]{Proposition}
   \newtheorem{lemma}[theorem]{Lemma}
   
   \newtheorem{claim}[theorem]{Claim}
   \theoremstyle{definition}
   \newtheorem{definition}[theorem]{Definition}
   \theoremstyle{remark}

 \usepackage{array}
 \usepackage[boxed]{algorithm2e}
 \usepackage{boxedminipage}
 \usepackage{hyperref}

 \newcommand{\nc}{\newcommand}

 \nc{\la}{\lambda}
 \nc{\etil}{\tilde{e}}
 \nc{\ftil}{\tilde{f}}
 \nc{\bftil}{\boldsymbol{\ftil}}
 \nc{\CrB}{\mathsf{B}}
 \nc{\Binf}{\CrB(\infty)}
 \nc{\Bla}{\CrB(\lambda)}
 \nc{\Binfv}{b_\infty}
 \nc{\CrT}{\mathsf{T}}
 \nc{\Tinf}{\CrT(\infty)}
 \nc{\Tla}{\CrT(\lambda)}
 \nc{\fwdtbl}{T}
 \nc{\Tinfv}{\fwdtbl_\infty}
 \nc{\fwdtblm}{\fwdtbl_{\fwdfxv}}
 \nc{\CrRT}{\bar{\mathsf{T}}}
 \nc{\Rinf}{\CrRT(\infty)}
 \nc{\Rla}{\CrRT(\lambda)}
 \nc{\rvstbl}{\bar{T}}
 \nc{\Rinfv}{\rvstbl_\infty}
 \nc{\rvstblm}{\rvstbl_{\rvsfxv}}
 \nc{\CrRC}{\mathsf{RC}}
 \nc{\RCinf}{\CrRC(\infty)}
 \nc{\RCla}{\CrRC(\lambda)}
 \nc{\RCinfv}{rc_\infty}
 \nc{\CrSE}{\mathsf{R}}
 \nc{\Tvla}{\CrSE_{\lambda}}
 \nc{\tvla}{{r}_{\lambda}}

 \nc{\alen}[3]{\textup{len}^{{#3}}_{{#1},{#2}}}
 \nc{\arig}[2]{\textup{rig}^{{#2}}_{{#1}}}

 \nc{\fwdlen}[2]{\textup{len}^{{#2}}_{{#1}}}
 \nc{\fwdrig}[2]{\textup{rig}^{{#2}}_{{#1}}}
 \nc{\fwdfx}[2]{\mathsf{m}_{{#1},{#2}}}
 \nc{\fwdfxv}{\mathbf{m}}
 \nc{\fwdn}[3]{\mathsf{n}^{{#3}}_{{#1},{#2}}}
 \nc{\fwdnv}{\mathbf{n}}
 \nc{\fwdfxvset}{\mathsf{M}_{\leq}}

 \nc{\rvslen}[2]{\bar{\textup{len}}^{{#2}}_{{#1}}}
 \nc{\rvsrig}[2]{\bar{\textup{rig}}^{{#2}}_{{#1}}}
 \nc{\rvsfx}[2]{\bar{\mathsf{m}}_{{#1},{#2}}}
 \nc{\rvsfxv}{\bar{\mathbf{m}}} 
 \nc{\rvsn}[3]{\bar{\mathsf{n}}^{{#3}}_{{#1},{#2}}}
 \nc{\rvsnv}{\bar{\mathbf{n}}}
 \nc{\rvsfxvset}{\mathsf{M}_{\geq}}

\begin{document}

\title{%
 Rigged Configuration Descriptions of the Crystals\\ 
 $\mathsf{B}(\infty)$ and $\mathsf{B}(\lambda)$ for Special Linear Lie Algebras}

\author[1]{Jin Hong}
\author[2]{Hyeonmi Lee}

\renewcommand\Affilfont{\normalsize}
\affil[1]{%
  Department of Mathematical Sciences\authorcr
  Seoul National University, Seoul 08826, Korea\authorcr
  \texttt{jinhong@snu.ac.kr}}

\affil[2]{%
  Department of Mathematics\authorcr
  Hanyang University, Seoul 04763, Korea\authorcr
  \texttt{hyeonmi@hanyang.ac.kr}}

\pagestyle{myheadings}

\date{}

\maketitle

\begin{abstract}\noindent
The rigged configuration realization $\RCinf$ of the crystal $\Binf$ was originally presented as a certain connected component within a larger crystal.
In this work, we make the realization more concrete by identifying the elements of~$\RCinf$ explicitly for the $A_n$-type case.
Two separate descriptions of~$\RCinf$ are obtained.
These lead naturally to isomorphisms $\RCinf\cong\Tinf$ and $\RCinf\cong\Rinf$, i.e., those with the marginally large tableau and  marginally large reverse tableau realizations of~$\Binf$, that may be computed explicitly.
We also present two descriptions of the irreducible highest weight crystal $\Bla$ in terms of rigged configurations.
These are obtained by combining our two descriptions of~$\RCinf$, the two mentioned isomorphisms, and two existing realizations of~$\Bla$ that were based on~$\Tinf$ and~$\Rinf$.

\medskip
\noindent{\sffamily
A version of this article has been published as J.\,Math.\,Phys. \textbf{58}, 101701, (2017).}\\
\url{https://doi.org/10.1063/1.4986276}
\end{abstract}

\section{Introduction}

The quantum group $U_q(\mathfrak{g})$ is a $q$-deformation of the universal enveloping algebra of a Lie algebra~$\mathfrak{g}$, and crystal bases capture the structure of $U_q(\mathfrak{g})$-modules in its most simplified form~\cite{Kas90,Kas91}.
As these $U_q(\mathfrak{g})$-modules are $q$-deformations of modules over the original Lie algebras, study of these structures can increase our knowledge of the Lie algebras.

The crystal~$\Binf$ is the crystal base of the negative part $U_q^-(\mathfrak{g})$ of a quantum group.
There are many realizations of~$\Binf$ that were obtained through various different approaches.
Some of these, such as~\cite{Kas01,KasSai97,Lit94,Lit98,NakZel97}, to name a few, are valid for arbitrary symmetrizable Kac-Moody algebras, and there are many other research results, each covering a smaller subset of Kac-Moody algebra types.
The goal of this work is to present an explicit description of the $A_n$-type crystal~$\Binf$, using certain combinatorial objects referred to as rigged configurations.

Rigged configurations~\cite{Bet31,KerKirRes86,KirRes86} were introduced as combinatorial objects parametrizing the eigenvectors of the Hamiltonian for certain integrable quantum systems.
The connection between the eigenvectors of the Hamiltonian and the highest weight vectors of a certain crystal was revealed in~\cite{OkaSchShi03b}, and this led to the concept of rigged configurations being used as a tool~\cite{OkaSakSch13,Sak14,Sch06,SalScr} for studying the structure of many crystals.

The realization~$\RCinf$ of~$\Binf$, given in terms of rigged configurations, was introduced in~\cite{SalScr,SalScr3}, for all symmetrizable Kac-Moody types.
To create~$\RCinf$, a crystal structure was first defined on a certain larger set by modifying the Kashiwara operator actions that were previously defined for a similar construction~\cite{Sch06}.
The set~$\RCinf$ was then presented as the connected component containing the multi-partition consisting only of empty parts, which is the rigged configuration that serves the role of the highest weight element.

The main contribution of this work is in making the realization~$\RCinf$ for the $A_n$-type crystal~$\Binf$ more concrete by explicitly describing the elements of~$\RCinf$.
To do this, we first recall that every element of~$\Binf$ may be reached from the highest weight element through a certain formatted sequence of lowering Kashiwara operator actions that is uniquely determined by the element.
We then compute the series of these actions on the highest weight rigged configuration within the crystal of arbitrary rigged configurations.
This provides us with an explicit listing, with no duplicates, of all the elements of~$\RCinf$.
Furthermore, we show how our computations may be reversed to obtain the unique formatted sequence of Kashiwara operator actions that leads to any given element of~$\RCinf$.
This allows one to distinguish elements of~$\RCinf$ from other elements of the larger crystal of all rigged configurations.
We also repeat all of the above with another family of formatted lowering Kashiwara operator sequences.

The first and second families of formatted sequences we use are closely related to the structures of the marginally large tableau realization $\Tinf$~\cite{HonLee08} and the marginally large reverse tableau realization $\Rinf$~\cite{Lee14} of~$\Binf$, respectively.
The connections are such that our two descriptions of~$\RCinf$ lead naturally to the crystal isomorphisms $\Tinf \cong \RCinf$ and $\Rinf \cong \RCinf$ that can be computed explicitly.

Note that a crystal isomorphism between $\Tinf$ and $\RCinf$ had appeared previously in~\cite{SalScr2}.
The isomorphism there was constructed by lifting the isomorphisms $\Tla \cong \RCla$, made available for each $\la\in P^+$, that involves bijections between rigged configurations and tensor products of Kirillov-Reshetikhin crystals, to the situation of~$\Tinf$ and~$\RCinf$. 
Here, $\Tla$ and $\RCla$ are the semistandard tableau and rigged configuration realizations of the highest weight irreducible crystal~$\Bla$.
This approach is very different from our direct isomorphism between $\Tinf$ and~$\RCinf$ that is presented after securing a concrete description of~$\RCinf$ as a set.

Another contribution of this work is in providing two new descriptions of~$\Bla$, the irreducible highest weight crystal.
It is known~\cite{Nak99} that $\Bla$ can be embedded in the tensor product of~$\Binf$ and a certain single-element crystal, so that $\Bla$ may essentially be seen as a subcrystal of~$\Binf$.
This was used in~\cite{Lee14} to obtain two realizations of~$\Bla$ that are essentially subsets of the two realizations $\Tinf$ and~$\Rinf$ of~$\Binf$.
Because our crystal isomorphisms $\Tinf \cong \RCinf$ and $\Rinf \cong \RCinf$ are so direct, we are able to translate the two realization of~$\Bla$ given by~\cite{Lee14} into those given in terms of rigged configurations.

We expect the approach of this paper to be extendable to other finite simple Lie algebra types and lead to rigged configuration descriptions of~$\Binf$ and~$\Bla$ for these types.

\section{Preliminary}

The rest of this paper will deal only with the $A_n$-type quantized universal enveloping algebra and its crystals.
Throughout this paper, Young tableaux will be displayed in the English notation with their rows numbered from top to bottom, so that the top row of a tableau is referred to as its first row.

Let $\fwdfxv = (\fwdfx{i}{j})_{i,j}$ be a collection of non-negative integers, where the indices span over the range $1 \leq j \leq n$ and $1 \leq i \leq n-j+1$, or, equivalently,
\begin{equation}
\label{eq:fwdfxrange}
1 \leq i, j \quad\text{and}\quad i+j \leq n+1.
\end{equation}
We will say that such an $\fwdfxv$ is \emph{weakly increasing} with respect to the index~$i$, if it satisfies
\begin{equation}
\label{eq:winc}
\fwdfx{i}{j} \leq \fwdfx{i+1}{j},
\end{equation}
for every meaningful choice of the indices~$i$ and~$j$.
The set of all weakly increasing~$\fwdfxv$ will be denoted by~$\fwdfxvset$.

To any collection $\fwdfxv = (\fwdfx{i}{j})_{i,j}$ of non-negative integers, whose indices cover the range~\eqref{eq:fwdfxrange}, we can associate the product of lowering Kashiwara operators
\begin{align}
\label{eq:exp1}
\bftil^{\,\fwdfxv} &= 
  \bftil^{\,\fwdfx{*}{n}}
  \cdots
  \bftil^{\,\fwdfx{*}{3}}
  \bftil^{\,\fwdfx{*}{2}}
  \bftil^{\,\fwdfx{*}{1}},\\
\intertext{where each}
\label{eq:exp2}
\bftil^{\,\fwdfx{*}{j}} &=
  \ftil_{n-j+1}^{\,\fwdfx{n-j+1}{j}}
  \cdots
  \ftil_{3}^{\,\fwdfx{3}{j}}
  \ftil_{2}^{\,\fwdfx{2}{j}}
  \ftil_{1}^{\,\fwdfx{1}{j}}.
\end{align}
Notice that the sequence formed by the Kashiwara operator indices appearing in~$\bftil^{\,\fwdfxv}$ is identical to the sequence of indices appearing in
\begin{equation}
\label{eq:redexp}
(s_1)(s_2 s_1) \cdots\cdots (s_{n-1}\cdots s_2 s_1)(s_n\cdots s_2 s_1),
\end{equation}
a reduced expression for the longest element of the $A_n$-type Weyl group.

The following restatement of a result appearing in~\cite{BerZel93,BerZel96} shows that the set $\fwdfxvset$ may be used as a labeling system for the elements of~$\Binf$.

\begin{lemma}
\label{lem:bij}
Let $\Binfv$ be the highest weight element of~$\Binf$.
The function that maps~$\fwdfxv$ to~$\bftil^{\,\fwdfxv} \Binfv$ is a bijection from~$\fwdfxvset$ to~$\Binf$.
\end{lemma}

Let us now briefly review some of the basic theory of the rigged configuration realization of crystal~$\Binf$.
Only the very minimum contents that are needed in this work will be covered.
What we provide below will be informal, and the reader is asked to consult the original papers~\cite{Sch06,SalScr} for the full precise definitions.
Note that we will not be dealing with the vacancy numbers in this paper, since they are redundant information in the rigged configuration model for~$\Binf$.

A rigged \emph{partition} may be understood to be a partition, i.e., a Young diagram, that has each of its rows labeled on the right by an integer.
A rigged partition is allowed to be empty and the row labels are referred to as the riggings.
A rigged \emph{configuration} of $A_n$-type is an ordered $n$-tuple of rigged partitions.

The lowering Kashiwara operator $\ftil_i$ acts on a rigged configuration by adding a box to its $i$-th component and changing some of the riggings.
One determines the smallest non-positive rigging present in the $i$-th rigged partition, adds a box to the longest row with the said rigging, and decrements the corresponding rigging by~$1$.
If there are no non-positive riggings, a new row consisting of a single box is created with the rigging~$-1$.
The only riggings that are affected by the $\ftil_i$ action, other than that of the row the box was added to, are those that belong to the $i$-th  rigged partition and its directly neighboring rigged partitions.
However, only those riggings that correspond to rows that are strictly longer than the (possibly empty) row the box was added to are changed.
The riggings of the longer rows within the $i$-th rigged partition are decremented by~$2$, and riggings of the longer rows within the neighboring rigged partitions are incremented by~$1$.
The full definition of the Kashiwara operator action is slightly more complicated than the description we gave here, but what we have explained should be sufficient for the purpose of this paper.

Let~$\RCinfv$ be the rigged configuration that consists of $n$ empty rigged partitions.
We have not explained the actions of the raising Kashiwara operator $\etil_i$, but $\RCinfv$ is a highest weight element.
The connected component within the crystal of rigged configurations containing $\RCinfv$ is denoted by~$\RCinf$.
The following result was given by~\cite{SalScr}.

\begin{lemma}
\label{lem:2}
The crystal~$\RCinf$ is isomorphic to~$\Binf$.
\end{lemma}

The next subjects we review are the marginally large tableau and the marginally large reverse tableau realizations of~$\Binf$.

A basic $i$-column (of $A_n$-type), defined for each $1\leq i \leq n$, is the highest weight element of the Young tableau realization~\cite{KasNak94} of~$\CrB(\Lambda_i)$.
In other words, this is a single column of boxes of height~$i$ such that the $j$-th box (from the top) contains~$j$, for $1\leq j \leq i$.
A marginally large tableau is a semi-standard tableau that contains precisely one basic $i$-column, for each $1\leq i \leq n$, among its columns.
Likewise, a marginally large reverse tableau is a semi-standard reverse tableau that contains precisely one basic $i$-column, for each $1\leq i \leq n$.
The general form of a marginally large tableau and a marginally large reverse tableau for $A_4$-type are given later in this paper by~\eqref{eq:fxT} and~\eqref{eq:mlrt}, respectively.
In this paper, the set of all marginally large tableaux and the set of all marginally large reverse tableaux will be denoted by~$\Tinf$ and~$\Rinf$, respectively.

The lowering Kashiwara operator $\ftil_i$ acts on a marginally large (reverse) tableau mostly as it would act on a normal semi-standard (reverse) tableau.
The only difference is that, when the resulting (reverse) tableau is no longer marginally large, a single basic $i$-column is inserted at the appropriate position to make it marginally large again.

It was shown in~\cite{Cli98,HonLee08,Rei97,Sai12,NakZel97,Lee14} that both $\Tinf$ and~$\Rinf$ are crystals that are isomorphic to~$\Binf$.
The highest weight elements in these two realizations of~$\Binf$ are their respective smallest elements, i.e., the marginally large tableau and marginally large reverse tableau consisting of just the $n$ basic columns.

The final subject we review is an interpretation of the irreducible highest weight crystal $\Bla$ as a subset of the crystal~$\Binf$.
For each $\la \in P^+$, the crystal $\Tvla = \{\tvla\}$, consisting of a single element, is defined to have the following crystal structure.
\begin{equation}
\label{eq:sscrys}
\operatorname{wt}(\tvla) = \la, \quad
\varepsilon_i(\tvla)=-\la(h_i), \quad
\varphi_i(\tvla)=0, \quad
\etil_i(\tvla)=0, \quad
\ftil_i(\tvla)=0.
\end{equation}
For each $\la\in P^+$, the symbol $b_\la$ will denote the highest weight element of~$\Bla$.
The following result appeared in~\cite{Nak99}, and an essentially equal claim was also given earlier by~\cite{Kas95}.

\begin{lemma}
\label{lem:cryimb}
For each $\la \in P^+$, there exist a unique strict crystal embedding
\begin{equation*}
\Bla \hookrightarrow \Binf \otimes \Tvla,
\end{equation*}
that maps $b_\lambda$ to $\Binfv \otimes \tvla$.
\end{lemma}

The existence of the embedding given by this claim implies that the connected component in the crystal $\Binf\otimes\Tvla$ containing the element $\Binfv \otimes \tvla$ is isomorphic to~$\Bla$.
This connected component was made more explicit in~\cite{Rei97,Sai12,Lee14,Nak99} with $\Binf$ replaced by its concrete realizations $\Tinf$ and~$\Rinf$.

\section{$A_4$-type Example}
\label{sec:A4}

The rigged configuration realization~$\RCinf$ of the crystal~$\Binf$ is concrete in the sense that it makes explicit computations possible.
However, its presentation as a certain connected component within a larger crystal of rigged configurations makes direct access to its elements difficult.
In particular, the realization $\RCinf$ currently lacks a description that allows for its elements to be explicitly listed, and it is not yet possible to determine whether a given rigged configuration belongs to~$\RCinf$ without applying the Kashiwara operators.

The bijection given by Lemma~\ref{lem:bij} is a very useful tool in this situation.
One can expect to resolve the mentioned difficulties by computing and collecting the elements $\bftil^{\,\fwdfxv} \RCinfv$ for all $\fwdfxv\in\fwdfxvset$.
This is precisely what will be done in the next section.
The $A_4$-type computation is provided in this section as an example so that it is easier to follow through the computations of the next section.

The rigged configuration $\ftil_1^{\,\fwdfx{1}{1}} \RCinfv$, which consists of four rigged partitions, is as follows.
\begin{center}
\includegraphics{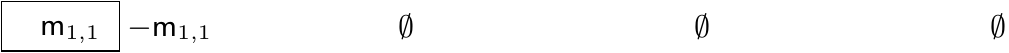}
\end{center}
The first rigged partition consists of a single row of boxes, and the other three rigged partitions are empty, except that the first partition would also be empty, when $\fwdfx{1}{1}=0$.
The length of the first partition is given by the number written within the rectangle and its rigging is as given to the right of the rectangle.
Each application of $\ftil_1$ adds one box to the row and decrements the rigging by~$1$.
There are no other rows whose rigging can be affected by the action, and the rigging of the only existing row stays negative (non-positive), so that the next application of $\ftil_1$ occurs in the same manner.

The result of computing $\bftil^{\,\fwdfx{*}{1}} \RCinfv$ is the following rigged configuration.
\begin{center}
\includegraphics{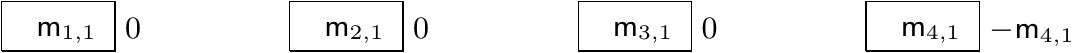}
\end{center}
Each of the four partitions consists of a single row (or could be empty).
Every $\ftil_i$ action adds one box to the $i$-th part and decrements the rigging of its only row by~$1$.
The action can also increment the rigging of the neighboring $(i-1)$-th partition row by~$1$, but this happens only if the row is longer than the $i$-th partition row that is having a box added to itself.
The weakly increasing property
$\fwdfx{1}{1} \leq \fwdfx{2}{1} \leq \fwdfx{3}{1} \leq \fwdfx{4}{1}$,
coupled with the order of the actions, imply that the incremental growth of the $(i-1)$-th partition rigging that starts from $-\fwdfx{i-1}{1}$ always reaches~$0$ and then stops.
Thus, the first wave of Kashiwara operators $\bftil^{\,\fwdfx{*}{1}}$ leaves the rigged configuration in the above simple state that has the first three riggings set to~$0$.

The result of applying $\bftil^{\,\fwdfx{*}{2}}$ to the above rigged configuration is as given below, where
\begin{align*}
\fwdn{2}{1}{1} &= \min\big\{\fwdfx{1}{2}, \fwdfx{2}{1}-\fwdfx{1}{1} \big\},\\
\fwdn{3}{1}{1} &= \min\big\{\fwdfx{2}{2}, \fwdfx{3}{1}-(\fwdfx{2}{1}-\fwdn{2}{1}{1})\big\},\\
\fwdn{4}{1}{1} &= \min\big\{\fwdfx{3}{2}, \fwdfx{4}{1}-(\fwdfx{3}{1}-\fwdn{3}{1}{1})\big\}.
\end{align*}
\begin{center}
\includegraphics{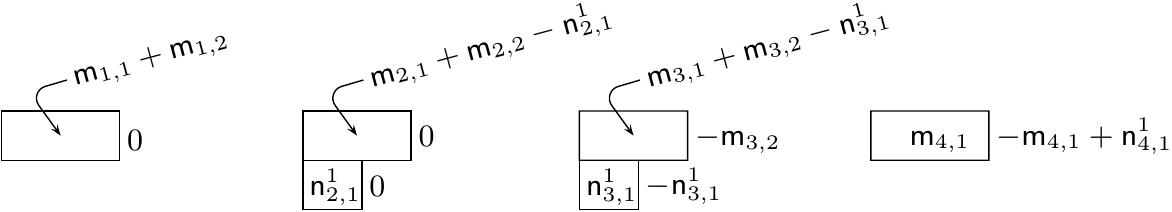}
\end{center}
The second wave of Kashiwara operator applications begins with the $\fwdfx{1}{2}$-many $\ftil_1$-actions increasing the length of the only row of the first partition and also decreasing its rigging.
This series of $\ftil_1$ actions also increases the rigging of the second partition, which is of length $\fwdfx{2}{1}$ at that time.
The second part rigging that is initially~$0$ has the possibility of reaching~$\fwdfx{1}{2}$, the number of $\ftil_1$-actions.
However, since the increment of the rigging lasts only while the row being lengthened is shorter than the row whose rigging is being affected, the rigging growth must be bounded by the difference in row lengths $\fwdfx{2}{1}-\fwdfx{1}{1}$, which is non-negative by the weakly increasing assumption.
This leads naturally to the introduction of the number~$\fwdn{2}{1}{1}$, which is defined to be the minimum of the two mentioned numbers.

After the $\ftil_1$-actions, the rigging $\fwdn{2}{1}{1}$ on the second partition is often strictly positive.
This forces the $\ftil_2$-actions into creating an additional row on the second partition.
Each $\ftil_2$-action reduces the riggings of the first row and the newly created second row by~$2$ and~$1$, respectively.
Since $\fwdn{2}{1}{1} \leq \fwdfx{1}{2} \leq \fwdfx{2}{2}$, the $\ftil_2$-actions eventually returns to lengthening the first row of the second partition, as the riggings of both rows reach $-\fwdn{2}{1}{1}$.

One can also verify that the combined effect on the first partition of all the $\ftil_2$-actions is to raise its rigging $-\fwdfx{1}{2}$ back to~$0$.
To see this, note that $\fwdn{2}{1}{1} \leq \fwdfx{1}{2} \leq \fwdfx{1}{1} + \fwdfx{1}{2}$, so that the rigging of the first partition reaches $-\fwdfx{1}{2}+\fwdn{2}{1}{1}$, as the second row of the second partition grows to its full length of~$\fwdn{2}{1}{1}$.
Now, the length of the first row of the first partition is possibly larger than that of the first row of the second partition by
$\max\{0, (\fwdfx{1}{1}+\fwdfx{1}{2})-\fwdfx{2}{1}\} = \fwdfx{1}{2} - \fwdn{2}{1}{1}$,
so that the remaining $(\fwdfx{2}{2}-\fwdn{2}{1}{1})$-many $\ftil_2$-actions bring the rigging of the first partition to
$(-\fwdfx{1}{2}+\fwdn{2}{1}{1}) + \min\{ \fwdfx{2}{2}-\fwdn{2}{1}{1}, \fwdfx{1}{2} - \fwdn{2}{1}{1} \} = 0$.

Let us also consider the effect of the $\ftil_2$-actions on the rigging of the third partition.
We first observe that $\fwdn{2}{1}{1} \leq \fwdfx{2}{1}-\fwdfx{1}{1} \leq \fwdfx{2}{1} \leq \fwdfx{3}{1}$, so that the rigging increases from~$0$ to~$\fwdn{2}{1}{1}$, while the second row of the second partition grows to its full length.
The remaining  $(\fwdfx{2}{2}-\fwdn{2}{1}{1})$-many $\ftil_2$-actions bring the rigging of the third partition to
$\fwdn{2}{1}{1} + \min\{ \fwdfx{2}{2}-\fwdn{2}{1}{1}, \fwdfx{3}{1} - \fwdfx{2}{1} \} = \fwdn{3}{1}{1}$.

After similarly handling the $\fwdfx{3}{2}$-many $\ftil_3$-actions, one finds that the second wave of Kashiwara operators $\bftil^{\,\fwdfx{*}{2}}$ returns all the riggings of the first two rigged partitions back to~$0$.

Further applications of all the Kashiwara operators contained in the remaining $\bftil^{\,\fwdfx{*}{3}}$ and~$\bftil^{\,\fwdfx{*}{4}}$ produce the rigged configuration given below, where
\begin{align*}
\fwdn{2}{2}{1} &= \min\big\{
  \fwdfx{1}{3},
   (\fwdfx{2}{1}+\fwdfx{2}{2}-\fwdn{2}{1}{1})
  -(\fwdfx{1}{1}+\fwdfx{1}{2})
\big\},\\
\fwdn{3}{2}{1} &= \min\big\{
 \fwdfx{2}{3},
  (\fwdfx{3}{1}+\fwdfx{3}{2}-\fwdn{3}{1}{1})
 -(\fwdfx{2}{1}+\fwdfx{2}{2}-\fwdn{2}{1}{1}-\fwdn{2}{2}{1})
\big\},\\
\fwdn{2}{3}{1} &= \min\big\{
 \fwdfx{1}{4},
  (\fwdfx{2}{1}+\fwdfx{2}{2}+\fwdfx{2}{3}-\fwdn{2}{1}{1}-\fwdn{2}{2}{1})
 -(\fwdfx{1}{1}+\fwdfx{1}{2}+\fwdfx{1}{3})
\big\},\\
\fwdn{3}{1}{2} &= \min\big\{
 \fwdn{2}{2}{1},
 \fwdn{3}{1}{1}-\fwdn{2}{1}{1}
\big\}.
\end{align*}
\begin{center}
\includegraphics{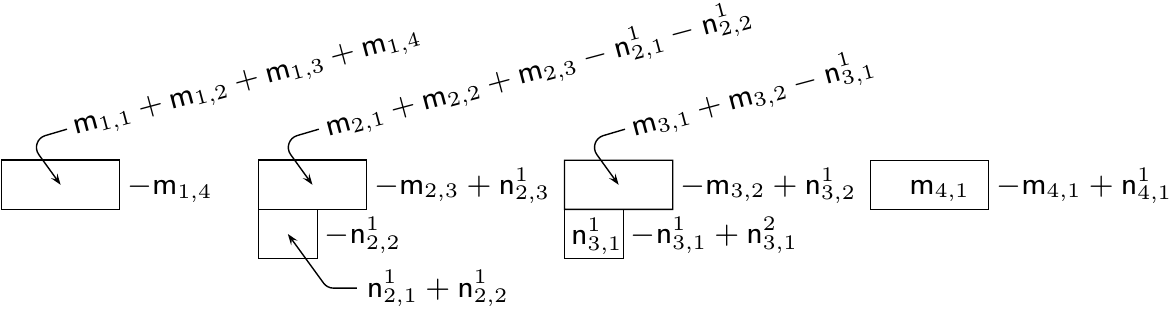}
\end{center}
As before, each $\fwdn{i}{j}{k} = \min\{\ ,\ \}$ definition has its origin in a comparison between a count of $\ftil_i$-actions and a difference of row lengths.

\section{Listing the Elements of $\RCinf$}

Let us now turn to expressing the rigged configuration $\bftil^{\,\fwdfxv} \RCinfv$, for the general $A_n$-type.
As in the previous section, our computations will be carried out in the most straightforward step by step manner.
Although this requires a lengthy and complicated inductive argument, the special structures of~$\fwdfxv$ and $\bftil^{\,\fwdfxv}$ make the computation of the general $\bftil^{\,\fwdfxv} \RCinfv$ feasible.

\begin{algorithm}
 \For{$j=1,2,\dots,n$}{
   \For{$i=1,2,\dots,n-j+1$}{
     $\fwdn{i}{j}{0} \leftarrow \fwdfx{i}{j}$\;
   }
 } 
 \For{$k=1,2,\dots,n$}{
   \For{$j=1,2,\dots,n-k+1$}{
     $\fwdn{1}{j}{k} \leftarrow 0$\;
     \For{$i=2,3,\dots,n-j-k+2$}{
       $\fwdn{i}{j}{k} \leftarrow \min\left\{
          \fwdn{i-1}{j+1}{k-1},
            \sum_{x=1}^{j  } \fwdn{i}{x}{k-1}
           -\sum_{x=1}^{j-1} \fwdn{i}{x}{k}
           -\sum_{x=1}^{j  } \fwdn{i-1}{x}{k-1}
           +\sum_{x=1}^{j  } \fwdn{i-1}{x}{k}
       \right\}$\;
     }
   }
 }
\caption{Generation of $\fwdnv = (\fwdn{i}{j}{k})_{i,j,k}$ from $\fwdfxv = (\fwdfx{i}{j})_{i,j}$}\label{alg:1}
\end{algorithm}

We first extend each $\fwdfxv = (\fwdfx{i}{j})_{i,j}\in\fwdfxvset$ into a certain larger collection of integers $\fwdnv = (\fwdn{i}{j}{k})_{i,j,k}$ through the iterative assignments\footnote{Brief explanation for those completely new to algorithms: Each ``$\leftarrow$'' specifies for its right-hand side value to be assigned to its left-hand side name. Statements are to be executed in the top-to-bottom order, except that each \textbf{for}-\textbf{end} construction specifies for its contents to be repeatedly executed.
These repetitions are to be carried out each time with the indices $i$, $j$, and~$k$ fixed to one of the values listed to their right, sequentially in the order listed.} given by Algorithm~\ref{alg:1}.
It is easy to check that the order of the assignments is such that all terms appearing on the right-hand side of the main assignment already has values assigned to it at the time of the assignment.
The indices covering the range
\begin{equation}
\label{eq:ijkrange}
1 \leq i,j,k \quad\text{and}\quad i+j+k \leq n+2
\end{equation}
correspond to the (strictly) extended part.
In the rest of this section, we assume that $\fwdfxv = (\fwdfx{i}{j})_{i,j}$ is a collection of non-negative integers that is weakly increasing with respect to the index~$i$ and that $\fwdnv = (\fwdn{i}{j}{k})_{i,j,k}$ is its extension obtained through Algorithm~\ref{alg:1}.

We ask the reader to read Definition~\ref{def:lenrig} and Theorem~\ref{thm:fwdgenA}, given at the end of this section, before continuing.
A lengthy computation will follow below, and it would be helpful to know where we are headed.

Recall that many inequalities were checked during our computation of~$\bftil^{\,\fwdfxv} \RCinfv$, for the $A_4$-type.
Below, we provide some properties of the extended~$\fwdnv$ that are essential, when one is working out $\bftil^{\,\fwdfxv} \RCinfv$ for the general $A_n$-type.

\begin{lemma}\label{lem:fwdD1}
We have $\fwdn{i}{j}{k} \leq \fwdn{i-1}{j+1}{k-1}$ for every meaningful choice of the indices.
\end{lemma}
\begin{proof}
This is a trivial consequence of the definition
\begin{equation*}
\fwdn{i}{j}{k} = \min\left\{
  \textstyle
  \fwdn{i-1}{j+1}{k-1},
  \sum_{x=1}^{j  } \fwdn{i}{x}{k-1}
 -\sum_{x=1}^{j-1} \fwdn{i}{x}{k}
 -\sum_{x=1}^{j  } \fwdn{i-1}{x}{k-1}
 +\sum_{x=1}^{j  } \fwdn{i-1}{x}{k}
\right\}
\end{equation*}
and is true regardless of any properties of~$\fwdfxv$.
\end{proof}

\begin{lemma}\label{lem:fwdD2}
We have $\fwdn{i}{j}{k} \leq \fwdn{i+1}{j}{k}$ for every meaningful choice of the indices.
In particular, each $\fwdn{i}{j}{k}$ is a non-negative integer.
\end{lemma}
\begin{proof}
Both the  inequality and the claim of non-negativity are true for the $k=0$ case, by our assumption on~$\fwdfxv$.
Let us take $0 \leq \fwdn{i}{j}{k} \leq \fwdn{i+1}{j}{k}$, for all possible $i$ and $j$ under a fixed~$k$, to be our induction hypothesis and work to show $0 \leq \fwdn{i}{j}{k+1} \leq \fwdn{i+1}{j}{k+1}$.

When $j\neq 1$, by the definition of $\fwdn{i+1}{j-1}{k+1}$, we have
\begin{equation*}
  \textstyle
  \sum_{x=1}^{j-1} \fwdn{i+1}{x}{k  } 
 -\sum_{x=1}^{j-2} \fwdn{i+1}{x}{k+1}
 -\sum_{x=1}^{j-1} \fwdn{i  }{x}{k  }
 +\sum_{x=1}^{j-1} \fwdn{i  }{x}{k+1}
\geq \fwdn{i+1}{j-1}{k+1},
\end{equation*}
so that
\begin{align*}
\fwdn{i+1}{j}{k+1}
&= \min\left\{
   \fwdn{i}{j+1}{k},
   \textstyle
     \sum_{x=1}^{j  } \fwdn{i+1}{x}{k  } 
    -\sum_{x=1}^{j-1} \fwdn{i+1}{x}{k+1}
    -\sum_{x=1}^{j  } \fwdn{i  }{x}{k  }
    +\sum_{x=1}^{j  } \fwdn{i  }{x}{k+1}
   \right\}\\
&\geq \min\left\{
   \fwdn{i}{j+1}{k},
   \fwdn{i+1}{j  }{k  } -\fwdn{i  }{j  }{k  } +\fwdn{i  }{j  }{k+1}
   \right\},
\end{align*}
and the same inequality for the $j=1$ case,
\begin{align*}
\fwdn{i+1}{1}{k+1}
\geq \min\left\{
     \fwdn{i}{2}{k},
     \fwdn{i+1}{1}{k}-\fwdn{i}{1}{k} +\fwdn{i}{1}{k+1}
     \right\},
\end{align*}
is trivially true by the definition of~$\fwdn{i+1}{1}{k+1}$.

Now, when $i>1$, we can apply our induction hypothesis to both terms within the right-hand side $\min\{\ ,\ \}$-expression to obtain
\begin{equation*}
\fwdn{i+1}{j}{k+1}
\geq \min\left\{ \fwdn{i-1}{j+1}{k}, \fwdn{i  }{j  }{k+1} \right\}
= \fwdn{i  }{j  }{k+1},
\end{equation*}
where the final equality is a consequence of Lemma~\ref{lem:fwdD1}.
For the remaining $i=1$ case, since $\fwdn{1}{j}{k+1}$ is defined to be zero, we have
\begin{equation*}
\fwdn{2}{j}{k+1}
\geq \min\left\{
         \fwdn{1}{j+1}{k},
         \fwdn{2}{j}{k} -\fwdn{1}{j}{k} +\fwdn{1}{j}{k+1}
         \right\}
= \min\left\{ \fwdn{1}{j+1}{k}, \fwdn{2}{j}{k} -\fwdn{1}{j}{k} \right\},
\end{equation*}
and since our induction hypothesis states that both terms within the $\min\{\ ,\ \}$-expression are non-negative, we have
\begin{equation*}
\fwdn{2}{j}{k+1} \geq 0 = \fwdn{1}{j}{k+1}.
\end{equation*}
This completes the induction step.
\end{proof}

\begin{lemma}
\label{lem:fwdE}
We have $\fwdn{i}{j}{k} \leq \fwdn{i}{j+1}{k-1}$, for every meaningful choice of the indices.
\end{lemma}
\begin{proof}
Since the right-hand side is non-negative by Lemma~\ref{lem:fwdD2}, this is trivially true for $i=1$ by the definition $\fwdn{1}{j}{k} = 0$, given for $k > 0$.
For $i\neq 1$, the claim can be obtained by combining the inequalities of Lemma~\ref{lem:fwdD1} and Lemma~\ref{lem:fwdD2}.
\end{proof}

\begin{lemma}\label{lem:fwdC}
We have $\fwdn{i}{j}{k} = 0$, for every meaningful choice of the indices such that $i \leq k$.
\end{lemma}
\begin{proof}
The statement is vacuous for $k=0$ and trivially true for $k=1$.
We take our induction hypothesis to be that $\fwdn{i}{j}{k} = 0$, for $1 \leq i \leq k$, and set our goal to showing $\fwdn{i}{j}{k+1} = 0$, for $1 \leq i \leq k+1$.

The $i=1$ case $\fwdn{1}{j}{k+1}$ is zero by definition.
For $1 < i \leq k+1$, we have
\begin{equation*}
\fwdn{i}{j}{k+1} =
\min\left\{
  \textstyle
  \fwdn{i-1}{j+1}{k},
  \sum_{x=1}^{j  } \fwdn{i}{x}{k} 
 -\sum_{x=1}^{j-1} \fwdn{i}{x}{k+1}
 -\sum_{x=1}^{j  } \fwdn{i-1}{x}{k}
 +\sum_{x=1}^{j  } \fwdn{i-1}{x}{k+1}
\right\},
\end{equation*}
and the first term $\fwdn{i-1}{j+1}{k}$ appearing on the right-hand side must be zero by our induction hypothesis.
Since Lemma~\ref{lem:fwdD2} states that the left-hand side is non-negative, the minimum itself must be zero.
This completes the induction step.
\end{proof}

Let us now introduce the notation
\begin{equation}
\label{eq:ds07n}
\alen{i}{j}{k} = \sum_{x=1}^{j} \fwdn{i}{x}{k-1} - \sum_{x=1}^{j-1} \fwdn{i}{x}{k},
\end{equation}
for indices $i$, $j$, and~$k$ in the range~\eqref{eq:ijkrange}.
To make some of our later formula manipulations more uniform, we additionally define
\begin{equation}
\alen{i}{0}{k} = 0.
\end{equation}
Note that we are introducing $\alen{i}{j}{k}$ as a shorthand notation for a certain formula, and its interpretation as a certain length will come later.

The two properties of the formal symbols $\alen{*}{*}{*}$ we present below are easy to obtain directly, but further inequalities concerning $\alen{*}{*}{*}$ will only be obtained as part of a lengthy induction argument to be given later in this section.

\begin{lemma}\label{lem:di3sg}
We have $\alen{i}{j}{k} \leq \alen{i}{j+1}{k}$, for every meaningful choice of the indices.
In particular, each $\alen{i}{j}{k}$ is a non-negative integer.
\end{lemma}
This is an immediate consequence of Lemma~\ref{lem:fwdD2}, Lemma~\ref{lem:fwdE}, and the observation
\begin{equation}
\alen{i}{j}{k}
= \fwdn{i}{1}{k-1}
 +\sum_{x=1}^{j-1} \Big(\fwdn{i}{x+1}{k-1} -\fwdn{i}{x}{k}\Big).
\end{equation}

\begin{lemma}
\label{lem:ls9an}
We have $\alen{i}{j}{k} \leq \alen{i+1}{j}{k}$, for every meaningful choice of the indices.
\end{lemma}
\begin{proof}
The definition of $\fwdn{i+1}{j}{k}$ may be rewritten in the form
\begin{align*}
\fwdn{i+1}{j}{k}
&= \fwdn{i}{j}{k}
 +\min\left\{\textstyle
      \fwdn{i}{j+1}{k-1} -\fwdn{i}{j}{k},  
      \sum_{x=1}^{j  } \fwdn{i+1}{x}{k-1} 
     -\sum_{x=1}^{j-1} \fwdn{i+1}{x}{k  }
     -\sum_{x=1}^{j  } \fwdn{i  }{x}{k-1}
     +\sum_{x=1}^{j-1} \fwdn{i  }{x}{k  }
  \right\}\\
&= \fwdn{i-1}{j}{k}
 +\min\left\{\textstyle
      \fwdn{i}{j+1}{k-1} -\fwdn{i}{j}{k},  
      \alen{i+1}{j}{k} -\alen{i}{j}{k}
  \right\},
\end{align*}
so that
\begin{equation*}
\textstyle
\fwdn{i+1}{j}{k} -\fwdn{i}{j}{k}
\leq
\alen{i+1}{j}{k} -\alen{i}{j}{k}.
\end{equation*}
The non-negativity of the left-hand side follows from Lemma~\ref{lem:fwdD2} and implies the claim.
\end{proof}

We will now work with the following object, which specifies much more detail than just the lengths of rows.

\begin{definition}
\label{def:rc}
For each $1 \leq v\leq n$ and $1\leq u\leq n-v+1$, we define $rc_{u,v}$ to be the rigged configuration with the following properties:
\begin{enumerate}
\item
\label{rcdef1}
The height of the $i$-th rigged partition is at most as follows:
\begin{alignat*}{2}
&v     \ & &\text{for}\ 1\leq i \leq u,\\
&v-1   \ & &\text{for}\ u+1 \leq i \leq n-v+1,\\
&n-i+1 \ & &\text{for}\ n-v+2 \leq i \leq n.
\end{alignat*}
\item
\label{rcdef2}
The length of the $k$-th row of the $i$-th rigged partition is as follows:
\begin{alignat*}{2}
&\alen{i}{v-k+1}{k}   \ & &\text{for}\ 1\leq i \leq u,\\
&\alen{i}{v-k}{k}     \ & &\text{for}\ u+1 \leq i \leq n-v+1,\\
&\alen{i}{n-i-k+2}{k} \ & &\text{for}\ n-v+2 \leq i \leq n.
\end{alignat*}
\item
\label{rcdef3}
If $u< n-v+1$, the rigging of the $k$-th row of the $i$-th rigged partition is as follows:
\begin{alignat*}{4}
 &0                      \ & &\text{for}\ 1\leq i < u, \qquad&
-&\fwdn{i}{v-k+1}{k-1}   \ & &\text{for}\ i=u,\\
 &\fwdn{i}{v-k}{k}       \ & &\text{for}\ i=u+1, \qquad&
 &0                      \ & &\text{for}\ u+1 < i \leq n-v+1,\\
-&\fwdn{i}{n-i-k+2}{k-1} \ & &\text{for}\ i= n-v+2, \qquad&
-&\fwdn{i}{n-i-k+2}{k-1} +\fwdn{i}{n-i-k+2}{k}
                         \ & &\text{for}\ n-v+2 < i \leq n.
\end{alignat*}
\item
\label{rcdef4}
If $u= n-v+1$, the rigging of the $k$-th row of the $i$-th rigged partition is as follows:
\begin{gather*}
 0\qquad\text{for}\ 1\leq i < u, \qquad\qquad\qquad
-\fwdn{i}{n-i-k+2}{k-1}\qquad\text{for}\ i=u,\\
-\fwdn{i}{n-i-k+2}{k-1} +\fwdn{i}{n-i-k+2}{k}
                         \quad\text{for}\ u+1 = n-v+2 \leq i \leq n.
\end{gather*}
\end{enumerate}
For convenience of notation, we set $rc_{u,0} = \RCinfv$.
\end{definition}

The range of~$k$ appearing in Item-\ref{rcdef2} of this definition is to be restricted to the height specified by Item-\ref{rcdef1}.
This automatically makes the three $j$-position indices appearing in Item-\ref{rcdef2} positive and brings all the indices of $\alen{*}{*}{*}$ to within the range specified by~\eqref{eq:ijkrange}.
The range of~$k$ appearing in Item-\ref{rcdef3} and Item-\ref{rcdef4} must also be restricted likewise to reflect the maximum heights, and this restriction ensures that all the $\fwdn{*}{*}{*}$ terms appearing in the two items are meaningful.
Since a rigged configuration attaches riggings only to non-empty rows, the riggings specified by Item-\ref{rcdef3} and Item-\ref{rcdef4} are meaningless when the corresponding rows are of length zero.
However, to make our formulas more uniform, we will include the (fake) riggings corresponding to the possibly empty rows in our computations.

At this point, we cannot claim that the description given by Definition~\ref{def:rc} properly defines a rigged configuration.
For example, it is not even clear if the set of $\alen{*}{*}{*}$ values suggested for any one rigged partition are actually appropriate as row lengths of a partition, i.e., whether they satisfy a certain weakly decreasing property.
However, we will inductively show that the result of partially computing $\bftil^{\,\fwdfxv} \RCinfv$ up to the point
\begin{equation}
\left(
  \ftil_{u}^{\,\fwdfx{u}{v}}
  \cdots
  \ftil_{2}^{\,\fwdfx{2}{v}}
  \ftil_{1}^{\,\fwdfx{1}{v}}\right)
  \bftil^{\,\fwdfx{*}{v-1}}
  \cdots
  \bftil^{\,\fwdfx{*}{2}}
  \bftil^{\,\fwdfx{*}{1}} \RCinfv
\end{equation}
must be as described by Definition~\ref{def:rc}, so that the existence of a rigged configuration matching the above description is automatically guaranteed.

Substituting $(u,v) = (1,1)$ into Definition~\ref{def:rc}, we find that $rc_{1,1}$ is defined to be the rigged configuration of the following specifications.
\begin{enumerate}
\item The first rigged partition consists of at most one row, and all the other parts are empty rigged partitions.
\item The length of the first row of the first rigged partition is $\alen{1}{1}{1} = \fwdn{1}{1}{0} = \fwdfx{1}{1}$.
\item The rigging of the first row of the first rigged partition is $-\fwdn{1}{1}{0} = -\fwdfx{1}{1}$, (assuming the row is non-empty).
\end{enumerate}
With the experience gained in the previous section through the $A_4$-type example, we know that $\ftil_{1}^{\,\fwdfx{1}{1}}\RCinfv$ is precisely the rigged configuration described above, even in the general $A_n$-type case.
This provides us with the base case of our induction argument.

We now fix integers $p$ and $q$ such that $1\leq q\leq n$ and $1\leq p\leq n-q+1$.
Our induction hypothesis will be the following.

\medskip
\noindent
\begin{boxedminipage}{\linewidth}
\textbf{Induction Hypothesis}:
For indices $u$ and $v$ such that
\begin{equation}
\label{eq:range}
\begin{aligned}
&\text{(a)~$1\leq v< q$ and $1\leq u \leq n-v+1$}\quad\text{and}\\
&\text{(b)~$v=q$ and $1\leq u < p$}
\end{aligned}
\end{equation}
the description of $rc_{u,v}$ given in Definition~\ref{def:rc} does correspond to a true rigged configuration and
\begin{equation}
\label{eq:indhyp}
\left(
  \ftil_{u}^{\,\fwdfx{u}{v}}
  \cdots
  \ftil_{2}^{\,\fwdfx{2}{v}}
  \ftil_{1}^{\,\fwdfx{1}{v}}\right)
  \bftil^{\,\fwdfx{*}{v-1}}
  \cdots
  \bftil^{\,\fwdfx{*}{2}}
  \bftil^{\,\fwdfx{*}{1}} \RCinfv
= rc_{u,v}.
\end{equation}
\end{boxedminipage}
\smallskip

\noindent
The base case that has been verified is the $(p,q) = (2,1)$ case of this Induction Hypothesis, and the goal of our induction step will be to show the following.

\medskip
\noindent
\begin{boxedminipage}{\linewidth}
\textbf{Goal of Induction Step}:
If we set
\begin{equation*}
rc_- = \begin{cases}
rc_{p-1,q}
 = \left(\ftil_{p-1}^{\,\fwdfx{p-1}{q}} \cdots \ftil_{1}^{\,\fwdfx{1}{q}}\right)
   \bftil^{\,\fwdfx{*}{q-1}} \cdots \bftil^{\,\fwdfx{*}{1}} \RCinfv,
 & \text{for $p\neq 1$,}\\
rc_{n-q+2,q-1}
 = \bftil^{\,\fwdfx{*}{q-1}} \cdots \bftil^{\,\fwdfx{*}{1}} \RCinfv,
 & \text{for $p=1$,}
\end{cases}
\end{equation*}
then $\ftil_p^{\,\fwdfx{p}{q}} rc_- = rc_{p,q}$.
\end{boxedminipage}
\smallskip

\noindent
A description of the rigged configuration $\bftil^{\,\fwdfxv} \RCinfv = rc_{1,n}$ will come as an immediate consequence of this induction argument.

By agreeing to the Induction Hypothesis, one is accepting that the non-negative integer $\alen{u}{v-k+1}{k}$, computed according to the formula~\eqref{eq:ds07n}, is the length of the $k$-th row of the $u$-th part of the rigged configuration~\eqref{eq:indhyp}, for each $u$ and~$v$ satisfying~\eqref{eq:range} and $1\leq k\leq v$.
Since the only $\ftil_u$ operators contained in $\bftil^{\,\fwdfx{*}{v}}$ already appear in~\eqref{eq:indhyp}, the non-negative integer $\alen{u}{v-k+1}{k}$ is also the length of the $k$-th row of the $u$-th part from $\bftil^{\,\fwdfx{*}{v}}\cdots \bftil^{\,\fwdfx{*}{1}} \RCinfv$.
Hence, one consequence of our induction argument will be that the following claim is valid for all meaningful choice of the indices (such that $j\neq 0$).

\begin{claim}\label{claim}
The non-negative integer $\alen{x}{y}{z}$ is the length of the $z$-th row of the $x$-th part of the rigged configuration
$\bftil^{\,\fwdfx{*}{y+z-1}}
 \cdots
 \bftil^{\,\fwdfx{*}{2}}
 \bftil^{\,\fwdfx{*}{1}} \RCinfv$.
\end{claim}

Before working on the induction step itself, we need to supplement our existing list of inequalities involving $\fwdn{*}{*}{*}$ and $\alen{*}{*}{*}$ with even more inequalities involving $\alen{*}{*}{*}$.
Unlike our previous lemmas, the inequalities to be given below will be valid only for a certain set of indices dictated by the Induction Hypothesis, until we are done with the induction argument.

\begin{lemma}
\label{lem:cr4s1}
If the interpretations of $\alen{i}{j}{k+1}$ and $\alen{i}{j}{k}$ given by Claim~\ref{claim} are correct, then $\alen{i}{j}{k+1} \leq \alen{i}{j}{k}$.
\end{lemma}
\begin{proof}
Recall that one may roughly view an $\ftil_i$ action as the adding of a box to the row of the smallest rigging.
With the addition of a box, the rigging for that row is reduced by one, the riggings of all the longer (upper) rows are reduced by two, and the riggings of all the shorter (lower) rows are not changed.
Hence, if an $\ftil_i$ action adds a box to a certain row, the immediately following $\ftil_i$ action, for the same~$i$, must materialize on either the same row or an upper row.
In other words, a series of consecutive $\ftil_i$ actions, for some fixed index~$i$, will always proceed (weakly) upwards among the rows within the $i$-th part of a rigged configuration.
Note that this argument remains valid even in the absence of non-positive riggings.

Taking for granted that the crystal structure on the rigged configurations is well-defined, the claim states that, during the application of $\ftil_i^{\fwdfx{i}{j+k}}$, the length of the $(k+1)$-th row of the $i$-th part cannot grow past the length of the (upper) $k$-th row, which has yet to receive box additions at that point.
\end{proof}

\begin{lemma}
\label{lem:s2eda}
If the interpretations of $\alen{i}{j}{k+1}$ and $\alen{i}{j+1}{k}$ given by Claim~\ref{claim} are correct, then $\alen{i}{j}{k+1} \leq \alen{i}{j+1}{k}$.
\end{lemma}
\begin{proof}
The assumption of this claim states that the two given integers are lengths of rows found within the same partition.
As such, the upper row must be at least as long as the lower row.
\end{proof}

\begin{lemma}
\label{lem:dvlke}
If the interpretations of $\alen{i-1}{2}{k}$, \dots, $\alen{i-1}{j}{k}$ and $\alen{i-1}{2}{k-1}$, \dots, $\alen{i-1}{j}{k-1}$ given by Claim~\ref{claim} are correct, then
$\alen{i}{j}{k} \leq \alen{i-1}{j+1}{k-1}$.
\end{lemma}
\begin{proof}
The proof will be an induction on~$j$.
The base case of $j=1$ is the claim
\begin{equation*}
\alen{i}{1}{k} = \fwdn{i}{1}{k-1}
\leq
\fwdn{i-1}{1}{k-2} +\fwdn{i-1}{2}{k-2} -\fwdn{i-1}{1}{k-1} = \alen{i-1}{2}{k-1}.
\end{equation*}
Since Lemma~\ref{lem:fwdD1} states that $\fwdn{i}{1}{k-1} \leq \fwdn{i-1}{2}{k-2}$, it suffices to show $0 \leq \fwdn{i-1}{1}{k-2} -\fwdn{i-1}{1}{k-1}$.
To see this, we infer from the definition of $\fwdn{i}{j}{k}$ that
$\fwdn{i}{1}{k} \leq \fwdn{i}{1}{k-1} -\fwdn{i-1}{1}{k-1} +\fwdn{i-1}{1}{k}$, rewrite this in the form $\fwdn{i-1}{1}{k-1} -\fwdn{i-1}{1}{k} \leq \fwdn{i}{1}{k-1} -\fwdn{i}{1}{k}$, and use Lemma~\ref{lem:fwdC} to conclude $0 \leq \fwdn{i}{1}{k-1} -\fwdn{i}{1}{k}$,
for every $1\leq i\leq n$.

Let us now move on to the induction step, setting $j>1$.
Using the notation~\eqref{eq:ds07n} we can express Lemma~\ref{lem:fwdD1} in the form
\begin{align*}
\fwdn{i}{j}{k-1}
\leq \fwdn{i-1}{j+1}{k-2}
= \fwdn{i-1}{j}{k-1} +\fwdn{i-1}{j+1}{k-2} -\fwdn{i-1}{j}{k-1}
= \fwdn{i-1}{j}{k-1} +\alen{i-1}{j+1}{k-1} -\alen{i-1}{j}{k-1}
\end{align*}
and rewrite the definition of $\fwdn{i}{j}{k}$ in the form
\begin{align*}
\fwdn{i}{j-1}{k}
&= \fwdn{i-1}{j}{k-1}
 +\min\left\{\textstyle
      0,
      \sum_{x=1}^{j-1} \fwdn{i  }{x}{k-1} 
     -\sum_{x=1}^{j-2} \fwdn{i  }{x}{k  }
     -\sum_{x=1}^{j  } \fwdn{i-1}{x}{k-1}
     +\sum_{x=1}^{j-1} \fwdn{i-1}{x}{k  }
  \right\}\\
&= \fwdn{i-1}{j}{k-1}
  +\min\left\{ 0, \alen{i}{j-1}{k} -\alen{i-1}{j}{k}\right\}.
\end{align*}
Combining these two relations, we find
\begin{align*}
\alen{i}{j}{k}
&=
 \fwdn{i}{j}{k-1} -\fwdn{i}{j-1}{k}
+\alen{i}{j-1}{k}\\
&\leq
 \alen{i-1}{j+1}{k-1} -\alen{i-1}{j}{k-1}
-\min\left\{ 0, \alen{i}{j-1}{k} -\alen{i-1}{j}{k}\right\}
+\alen{i}{j-1}{k}\\
&=
 \alen{i-1}{j+1}{k-1} -\alen{i-1}{j}{k-1}
+\max\left\{ \alen{i-1}{j}{k}, \alen{i}{j-1}{k} \right\}.
\end{align*}
Finally, Lemma~\ref{lem:cr4s1}, which states $\alen{i-1}{j}{k} \leq \alen{i-1}{j}{k-1}$, and our current induction hypothesis $\alen{i}{j-1}{k} \leq \alen{i-1}{j}{k-1}$ together imply
\begin{equation*}
\max\left\{ \alen{i-1}{j}{k}, \alen{i}{j-1}{k} \right\}
\leq
\alen{i-1}{j}{k-1},
\end{equation*}
so that
\begin{equation*}
\alen{i}{j}{k}
\leq
  \left( \alen{i-1}{j+1}{k-1} -\alen{i-1}{j}{k-1} \right)
 +\alen{i-1}{j}{k-1}
= \alen{i-1}{j+1}{k-1}.
\end{equation*}
Note that the conditions required for the applications of Lemma~\ref{lem:cr4s1} at all steps of the induction process are provided by the assumptions of the current lemma.
This concludes the induction step.
\end{proof}

The following result will be used later to justify our treatment of the (fake) riggings corresponding to empty rows.

\begin{lemma}
\label{lem:f8s3k}
If the interpretations of $\alen{i}{j}{k}$, $\alen{i}{j-1}{k+1}$, \dots, $\alen{i}{1}{k+j-1}$ given by Claim~\ref{claim} are correct and $\alen{i}{j}{k} = 0$, then $\fwdn{i}{j}{k-1} = 0$ and $\fwdn{i}{j}{k} = 0$.
\end{lemma}
\begin{proof}
The $\alen{*}{*}{*}$ values appearing in this claim are assumed to be lengths of rows appearing within the same partition, and, as such, must be non-negative and weakly decreasing in the order given, i.e.,
\begin{equation*}
0 = \alen{i}{j}{k} \geq \alen{i}{j-1}{k+1} \geq \dots \geq \alen{i}{1}{k+j-1} \geq 0,
\end{equation*}
so that they are all zero.
Recalling both claims of Lemma~\ref{lem:fwdD2} and using definition~\eqref{eq:ds07n} to write
\begin{equation*}
\textstyle
0 = \alen{i}{j}{k} + \dots + \alen{i}{1}{k+j-1}
= \sum_{y=1}^{j} \fwdn{i}{y}{k-1}
\geq \sum_{y=1}^{j} \fwdn{x}{y}{k-1}
\geq 0,
\end{equation*}
for $x \leq i$, we can claim that
\begin{equation}\label{eq:s3s3f}
 \fwdn{x}{y}{k-1} = 0, \quad \text{for $x\leq i$ and $1\leq y \leq j$.}
\end{equation}
In particular, this gives us the first claim $\fwdn{i}{j}{k-1} = 0$ of this lemma.
Repeating the above argument with $0 = \alen{i}{j-1}{k+1} + \dots + \alen{i}{1}{k+j-1}$, one can also obtain
\begin{equation}\label{eq:r8s2a}
 \fwdn{x}{y}{k} = 0, \quad \text{for $x\leq i$ and $1\leq y \leq j-1$}.
\end{equation}

Now, the definition of $\fwdn{i}{j}{k}$ implies that
\begin{equation*}
\textstyle
\fwdn{x}{j}{k} \leq
 \sum_{y=1}^{j  } \fwdn{x}{y}{k-1}
-\sum_{y=1}^{j-1} \fwdn{x}{y}{k}
-\sum_{y=1}^{j  } \fwdn{x-1}{y}{k-1}
+\sum_{y=1}^{j  } \fwdn{x-1}{y}{k}
\end{equation*}
and, if $x \leq i$, most of these terms are zero by~\eqref{eq:s3s3f} and~\eqref{eq:r8s2a}, so that
\begin{equation*}
\fwdn{x}{j}{k} \leq \fwdn{x-1}{j}{k},
\quad\text{for every $x \leq i$.}
\end{equation*}
Hence, if $i > k$, then $0 \leq \fwdn{i}{j}{k} \leq \fwdn{i-1}{j}{k} \leq \cdots \leq \fwdn{k}{j}{k} = 0$, by Lemma~\ref{lem:fwdD2} and Lemma~\ref{lem:fwdC}, so that these are all zeros.
On the other hand, if $i \leq k$, then $\fwdn{i}{j}{k} = 0$, directly by Lemma~\ref{lem:fwdC}.
\end{proof}

We are now close to working on the induction step itself.
The next lemma does not depend on the Induction Hypothesis.

\begin{lemma}
\label{lem:f9e2a}
Let $l_1 \geq l_2 \geq \cdots l_t \geq l_{t+1} = 0$ and $m = r_0 \geq r_1 \geq \cdots \geq r_t \geq r_{t+1} = 0$ be non-negative integers such that $r_k = 0$ for every $l_k = 0$.
Let $rc_-$ be a rigged configuration whose $p$-th part is the rigged partition of the following properties\textup{:}
\begin{enumerate}
\item The height is at most~$t$.
\item The length of the $k$-th row is $l_k$.
\item The rigging of the $k$-th row is $r_k$, \textup{(}if $l_k \neq 0$\textup{)}.
\end{enumerate}
Then the $p$-th part of $\ftil_p^m rc_-$ is a rigged partition of the following form\textup{:}
\begin{enumerate}
\item The height is at most~$t+1$.
\item The length of the $k$-th row is $l_k +r_{k-1} -r_k$. In particular, the $(t+1)$-th row is of length~$r_t$.
\item The rigging of the $k$-th row is $-r_{k-1}$, \textup{(}if the row is non-empty\textup{)}.
\end{enumerate}
\end{lemma}
\begin{proof}
Let us first work with a small example under a slightly stronger set of conditions.
We fix strictly positive integers $r_0 > r_1 > r_2 > 0$ and consider the application of $\ftil_p^{r_0}$ to a rigged configuration whose $p$-th part is as follows:
(a)~It consists of two non-empty rows;
(b)~The riggings of the upper and lower rows are $r_1$ and $r_2$, respectively.

Since there are no non-positive riggings, the first $\ftil_p$ action will add one box to a new third row, set the rigging of this third row to~$-1$, and reduce the riggings of the two existing rows by~$2$ each.
Further applications of $\ftil_p$ must materialize on the third row, each of them reducing the rigging of the third row by~$1$ and reducing the riggings of the other two rows by~$2$, until the more rapidly decreasing rigging of the second row reaches the slowly decreasing rigging of the third row.
More precisely, each of the initial $r_2$-many applications of $\ftil_p$ adds a box to the third row, and the riggings of the three rows, listed from top to bottom, become $r_1-2r_2$, $-r_2$, and~$-r_2$.

As the three riggings satisfy $r_1-2r_2 > -r_2 = -r_2$, the next few applications of $\ftil_p$ must materialize on the second row.
Each additional such $\ftil_p$ action reduces the riggings of the three rows, listed from top to bottom, by $2$, $1$, and~$0$, respectively.
After $(r_1-r_2)$-many applications of $\ftil_p$ adding boxes to the second row, the riggings of the three row, listed from top to bottom, become $-r_1$, $-r_1$, and~$-r_2$.

The $\ftil_p$ actions then add boxes to the first row, reducing its rigging by~$1$ each time and not affecting the other two riggings.
As the bottom two rows have received $r_2+(r_1-r_2) = r_1$ applications of $\ftil_p$, the top row must receive the remaining $(r_0-r_1)$-many $\ftil_p$ actions, and its rigging becomes~$-r_0$.

In summary, the $r_0$-many applications of $\ftil_p$, increase the lengths of the three rows, including the bottom row which started out as empty, listed from top to bottom, by $r_0-r_1$, $r_1-r_2$, and~$r_2$.
The riggings of the three rows, listed from top to bottom, become $-r_0$, $-r_1$, and $-r_2$.
This is in full agreement with the contents of this claim.

One can now check that, with appropriate interpretations given to descriptions of the degenerate situations, the above arguments and the concluding summary remain valid with non-negative integers $r_0 \geq r_1 \geq r_2 \geq 0$ and a partition consisting of two non-empty rows.

Furthermore, when $r_2 = 0$, the process of creating a third row of length $r_2$ becomes vacuous and the above arguments and summary remain valid even if the second row was empty at the beginning.
Similarly, starting with an empty partition creates no problem if the starting riggings were given as $r_1=r_2=0$.
The $r_0=0$ case, which implies $r_1=r_2=0$ and no $\ftil_p$ action, is also compatible with the above description.

Finally, it is clear that what we have observed from this small two-row example generalizes easily to a rigged partition of arbitrary height.
\end{proof}

We are now ready to work on the induction step itself.
The next three lemmas describe the result of applying $\ftil_p^{\,\fwdfx{p}{q}}$ to the final element handled by the Induction Hypothesis.
The first of these presents the $p$-th part of the resulting rigged configuration and the two lemmas that follow discuss the riggings of the $(p-1)$-th and $(p+1)$-th parts.

\begin{lemma}
\label{lem:d9e2p}
Let us accept the \textup{Induction Hypothesis} and set
\begin{equation*}
rc_- = \begin{cases}
rc_{p-1,q}
 = \left(\ftil_{p-1}^{\,\fwdfx{p-1}{q}} \cdots \ftil_{1}^{\,\fwdfx{1}{q}}\right)
   \bftil^{\,\fwdfx{*}{q-1}} \cdots \bftil^{\,\fwdfx{*}{1}} \RCinfv,
 & \text{for $p\neq 1$,}\\
rc_{n-q+2,q-1}
 = \bftil^{\,\fwdfx{*}{q-1}} \cdots \bftil^{\,\fwdfx{*}{1}} \RCinfv,
 & \text{for $p=1$.}
\end{cases}
\end{equation*}
Then the $p$-th part of the rigged configuration $\ftil_{p}^{\,\fwdfx{p}{q}} rc_-$ is the rigged partition with the following properties\textup{:}
\begin{enumerate}
\item The height is at most~$q$.
\item The length of the $k$-th row is $\alen{p}{q-k+1}{k}$.
\item The rigging of the $k$-th row is $-\fwdn{p}{q-k+1}{k-1}$, \textup{(}if the row is non-empty\textup{)}.
\end{enumerate}
\end{lemma}
\begin{proof}
For $p\neq 1$, the Induction Hypothesis states that the $p$-th part of $rc_-$ is a rigged partition with the following properties: (a)~The height is at most $q-1$; (b)~The length of the $k$-th row is $\alen{p}{q-k}{k}$; (c)~The rigging of the $k$-th row is $\fwdn{p}{q-k}{k}$ (if the row is non-empty).
For $p=1$, the Induction Hypothesis states that the $1$-st part of $rc_-$ is a rigged partition with the following properties: (a)~The height is at most $q-1$; (b)~The length of the $k$-th row is $\alen{1}{q-k}{k}$; (c)~The rigging of the $k$-th row is~$0$ (if the row is non-empty).
The (c)-properties may seem different in the two cases, but since Lemma~\ref{lem:fwdC} states that $\fwdn{1}{q-k}{k} = 0$, we can view the $p=1$ case description of the $p$-th part of $rc_-$ as a special case of the $p\neq 1$ case description.

We know from Lemma~\ref{lem:fwdE} that the exponent $\fwdfx{p}{q}$ for $\ftil_p$ and the $q-1$ riggings satisfy
$\fwdfx{p}{q}
= \fwdn{p}{q}{0}
\geq \fwdn{p}{q-1}{1}
\geq \cdots
\geq \fwdn{p}{1}{q-1} \geq 0$.
We also know from Lemma~\ref{lem:f8s3k} that if the length $\alen{p}{q-k}{k}$ of a row is zero, then the corresponding (fake) rigging $\fwdn{p}{q-k}{k}$ is also zero.
Applying Lemma~\ref{lem:f9e2a} to this situation, we can claim that the $p$-th part of $\ftil_{p}^{\,\fwdfx{p}{q}} rc_-$ is a rigged partition with the following properties:
(a)~The height is at most~$q$; (b)~The length of the $k$-th row is $\alen{p}{q-k}{k}+\fwdn{p}{q-k+1}{k-1}-\fwdn{p}{q-k}{k} = \alen{p}{q-k+1}{k}$; (c)~The rigging of the $k$-th row is $-\fwdn{p}{q-k+1}{k-1}$.
\end{proof}

\begin{lemma}
\label{lem:c5k4s}
Let us accept the \textup{Induction Hypothesis} and set
\begin{equation*}
rc_- = rc_{p-1,q}
 = \left(\ftil_{p-1}^{\,\fwdfx{p-1}{q}} \cdots \ftil_{1}^{\,\fwdfx{1}{q}}\right)
   \bftil^{\,\fwdfx{*}{q-1}} \cdots \bftil^{\,\fwdfx{*}{1}} \RCinfv,
\end{equation*}
assuming $p\neq 1$.
Then the $(p-1)$-th part of the rigged configuration $\ftil_{p}^{\,\fwdfx{p}{q}} rc_-$ is the rigged partition with the following properties\textup{:}
\begin{enumerate}
\item The height is at most~$q$.
\item The length of the $k$-th row is $\alen{p-1}{q-k+1}{k}$.
\item The rigging of the $k$-th row is $0$, \textup{(}if the row is non-empty\textup{)}.
\end{enumerate}
\end{lemma}
\begin{proof}
For $p\neq 1$, the Induction Hypothesis states that the $(p-1)$-th part of $rc_-$ is a rigged partition with the following properties: (a)~The height is at most $q$; (b)~The length of the $k$-th row is $\alen{p-1}{q-k}{k}$; (c)~The rigging of the $k$-th row is $-\fwdn{p-1}{q-k+1}{k-1}$ (if the row is non-empty).
Since the $\ftil_p$ actions will not add boxes to or remove boxes from the $(p-1)$-th part, we already have confirmation of the first and second properties stated by this claim.

It remains to see the effect of $\ftil_{p}^{\,\fwdfx{p}{q}}$ action on the riggings of the $(p-1)$-th part.
Generalizing our experience of the $A_4$-type example discussed in Section~\ref{sec:A4}, we can state that the rigging of the $k$-th row of the $(p-1)$-th part from $\ftil_{p}^{\,\fwdfx{p}{q}} rc_-$ will be as follows:
\begin{equation}
\label{eq:c9i5j}
\begin{aligned}
-\fwdn{p-1}{q-k+1}{k-1}
&+\min\left\{
    \alen{p}{q}{1} -\alen{p}{q-1}{1},
    \max\left( 0, \alen{p-1}{q-k+1}{k} -\alen{p}{q-1}{1} \right)
 \right\}\\
&+\min\left\{
    \alen{p}{q-1}{2} -\alen{p}{q-2}{2},
    \max\left( 0, \alen{p-1}{q-k+1}{k} -\alen{p}{q-2}{2} \right)
  \right\}\\
&\ \vdots\\
&+\min\left\{
    \alen{p}{q-k+2}{k-1} -\alen{p}{q-k+1}{k-1},
    \max\left( 0, \alen{p-1}{q-k+1}{k} -\alen{p}{q-k+1}{k-1} \right)
  \right\}\\
&+\min\left\{
    \alen{p}{q-k+1}{k} -\alen{p}{q-k}{k},
    \max\left( 0, \alen{p-1}{q-k+1}{k} -\alen{p}{q-k}{k} \right)
  \right\}\\
&+\min\left\{
    \alen{p}{q-k}{k+1} -\alen{p}{q-k-1}{k+1},
    \max\left( 0, \alen{p-1}{q-k+1}{k} -\alen{p}{q-k-1}{k+1} \right)
  \right\}\\
&\ \vdots\\
&+\min\left\{
    \alen{p}{2}{q-1} -\alen{p}{1}{q-1},
    \max\left( 0, \alen{p-1}{q-k+1}{k} -\alen{p}{1}{q-1} \right)
  \right\}\\
&+\min\left\{
    \alen{p}{1}{q} -\alen{p}{0}{q},
    \max\left( 0, \alen{p-1}{q-k+1}{k} -\alen{p}{0}{q} \right)
  \right\}.
\end{aligned}
\end{equation}
Our use of the notation $\alen{p}{q-x+1}{x}$ ($1\leq x \leq q$) as lengths of rows is justified by Lemma~\ref{lem:d9e2p}, and the use of all other $\alen{*}{*}{*}$ terms is justified by the Induction Hypothesis.
The left item of each $\min\{\ ,\ \}$ term is the number of boxes added by $\ftil_{p}^{\,\fwdfx{p}{q}}$ to a row of the $p$-th rigged partition.
The addition of each box to the $p$-th part increases the rigging of a row belonging to the $(p-1)$-th part by~$1$ if and only if the row that is having the box added to is shorter than the row we are focusing on, and this condition is captured by the $\max\{0,\ \}$ term placed within each $\min\{\ ,\ \}$ term.
The $\max\{0,\ \}$ term measures how large the $k$-th row of the $(p-1)$-th partition is in comparison to the initial length of the row that is receiving the box additions.

We start the simplification of~\eqref{eq:c9i5j} by first studying the $k$-th line.
Being the number of boxes added by $\ftil_{p}^{\,\fwdfx{p}{q}}$, the left item of the $\min\{\ ,\ \}$ term must be non-negative.
We also know from Lemma~\ref{lem:ls9an} that $\alen{p}{q-k+1}{k} \geq \alen{p-1}{q-k+1}{k}$.
Making the further observation that the two terms being subtracted are identical, we can claim that the $k$-th line of~\eqref{eq:c9i5j} is equal to the right $\max\{0,\ \}$ term.
To simplify this further, we can rewrite the definition of~$\fwdn{i}{j}{k}$ in the form
\begin{align*}
\fwdn{p}{q-k}{k}
&= \min\left\{
  \textstyle
  \ \fwdn{p-1}{q-k+1}{k-1},\
  \begin{aligned}
  \fwdn{p-1}{q-k+1}{k-1} 
   +\left(\textstyle
    \sum_{x=1}^{q-k  } \fwdn{p}{x}{k-1} -\sum_{x=1}^{q-k-1} \fwdn{p}{x}{k}
    \right)&\\
  -\left(\textstyle
    \sum_{x=1}^{q-k+1} \fwdn{p-1}{x}{k-1} -\sum_{x=1}^{q-k  } \fwdn{p-1}{x}{k}
    \right)&
  \end{aligned}
\right\}\\
&= \fwdn{p-1}{q-k+1}{k-1}
  +\min\left\{
     0,
     \alen{p}{q-k}{k} -\alen{p-1}{q-k+1}{k}
   \right\}
\end{align*}
and move two terms to opposite sides to conclude that the $k$-th line of~\eqref{eq:c9i5j} is equal to
\begin{equation}
\label{eq:a3f3s}
\max\left( 0, \alen{p-1}{q-k+1}{k} -\alen{p}{q-k}{k} \right)
= \fwdn{p-1}{q-k+1}{k-1} -\fwdn{p}{q-k}{k}.
\end{equation}

Let us next treat the first $k-1$ lines of~\eqref{eq:c9i5j}.
Supplementing Lemma~\ref{lem:s2eda} to the arguments made for the $k$-th line, one can claim as before that the right $\max\{0,\ \}$ term is the smaller of the two terms appearing in each $\min\{\ ,\ \}$ term among these $k-1$ lines.
Now, combining Lemma~\ref{lem:cr4s1} and Lemma~\ref{lem:ls9an} we know
\begin{equation*}
\alen{p-1}{q-k+1}{k} \leq \alen{p-1}{q-k+1}{k-1} \leq \alen{p}{q-k+1}{k-1}
\end{equation*}
so that the $\max\{0,\ \}$ term of the $(k-1)$-th line must be zero.
Combining this further with Lemma~\ref{lem:s2eda} we can claim that every $\min\{\ ,\ \}$ term appearing in the first $k-1$ lines of~\eqref{eq:c9i5j} are zero.

It remains to handle the $(k+1)$-th through $q$-th lines of~\eqref{eq:c9i5j}.
Noting $\alen{p}{q-k}{k+1} \leq \alen{p-1}{q-k+1}{k}$ from Lemma~\ref{lem:dvlke}, we see that the $(k+1)$-th line of~\eqref{eq:c9i5j} is equal to the left item of the $\min\{\ ,\ \}$ term.
Combining this further with Lemma~\ref{lem:s2eda}, we see that the same may be said of all the $(k+1)$-th through $q$-th lines of~\eqref{eq:c9i5j}.
Recalling the definition~\eqref{eq:ds07n}, we can write these terms as follows:
\begin{equation}
\label{eq:m4i5o}
\begin{aligned}
\alen{p}{q-k}{k+1} -\alen{p}{q-k-1}{k+1}
&= \fwdn{p}{q-k}{k} -\fwdn{p}{q-k-1}{k+1}\\
&\ \vdots\\
\alen{p}{2}{q-1} -\alen{p}{1}{q-1}
&= \fwdn{p}{2}{q-2} -\fwdn{p}{1}{q-1}\\
\alen{p}{1}{q} -\alen{p}{0}{q}
&= \fwdn{p}{1}{q-1}
\end{aligned}
\end{equation}

Finally, adding all terms of~\eqref{eq:a3f3s} and~\eqref{eq:m4i5o} to the initial rigging $-\fwdn{p-1}{q-k+1}{k-1}$, we arrive at the claimed rigging of zero.
\end{proof}

\begin{lemma}
\label{lem:f3k4s}
Let us accept the \textup{Induction Hypothesis} and set 
\begin{equation*}
rc_- = \begin{cases}
rc_{p-1,q}
 = \left(\ftil_{p-1}^{\,\fwdfx{p-1}{q}} \cdots \ftil_{1}^{\,\fwdfx{1}{q}}\right)
   \bftil^{\,\fwdfx{*}{q-1}} \cdots \bftil^{\,\fwdfx{*}{1}} \RCinfv,
 & \text{for $p\neq 1$,}\\
rc_{n-q+2,q-1}
 = \bftil^{\,\fwdfx{*}{q-1}} \cdots \bftil^{\,\fwdfx{*}{1}} \RCinfv,
 & \text{for $p=1$.}
\end{cases}
\end{equation*}
Then the $(p+1)$-th part of the rigged configuration $\ftil_{p}^{\,\fwdfx{p}{q}} rc_-$ is the rigged partition with the following properties\textup{:}
\begin{enumerate}
\item The height is at most~$q-1$.
\item The length of the $k$-th row is $\alen{p+1}{q-k}{k}$.
\item If $p < n-q+1$, the rigging of the $k$-th row is $\fwdn{p+1}{q-k}{k}$, \textup{(}assuming the row is non-empty\textup{)}.
\item If $p = n-q+1$, the rigging of the $k$-th row is $-\fwdn{p+1}{q-k}{k-1} +\fwdn{p+1}{q-k}{k}$, \textup{(}assuming the row is non-empty\textup{)}.
\end{enumerate}
\end{lemma}
\begin{proof}
The proof is similar to that of Lemma~\ref{lem:c5k4s}, so we will be brief.
For $p\neq 1$, the Induction Hypothesis states that the $(p+1)$-th part of $rc_-$ is a rigged partition with the following properties:
(a)~The height is at most $q-1$;
(b)~The length of the $k$-th row is $\alen{p+1}{q-k}{k}$;
(c)~If $p < n-q+1$, the rigging of the $k$-th row is $0$ (assuming the row is non-empty);
(d)~If $p = n-q+1$, the rigging of the $k$-th row is $-\fwdn{p+1}{q-k}{k-1}$ (assuming the row is non-empty).
If $p=1$, the Induction Hypothesis states that the $2$-nd part of $rc_-$ is a rigged partition with the following properties:
(a)~The height is at most $q-1$;
(b)~The length of the $k$-th row is $\alen{2}{q-k}{k}$;
(c)~If $q < n$, the rigging of the $k$-th row is $0$ (assuming the row is non-empty);
(d)~If $q = n$, the rigging of the $k$-th row is $-\fwdn{2}{q-k}{k-1}$ (assuming the row is non-empty).
The $p=1$ case description of the $(p+1)$-th part of $rc_-$ may be viewed as a special case of the $p\neq 1$ case description.

The first two properties stated by this claim have already been confirmed and it remains to compute the riggings.
After the $\ftil_{p}^{\,\fwdfx{p}{q}}$ action, the rigging of the $k$-th row of the $(p+1)$-th part will be as follows:
\begin{equation}
\label{eq:mrDel}
\begin{aligned}
\left(0\quad\text{or}\quad -\fwdn{p+1}{q-k}{k-1} \right)
&+\min\left\{
    \alen{p}{q}{1} -\alen{p}{q-1}{1},
    \max\left( 0, \alen{p+1}{q-k}{k} -\alen{p}{q-1}{1} \right)
 \right\}\\
&+\min\left\{
    \alen{p}{q-1}{2} -\alen{p}{q-2}{2},
    \max\left( 0, \alen{p+1}{q-k}{k} -\alen{p}{q-2}{2} \right)
  \right\}\\
&\hspace{3.8pt}\vdots\\
&+\min\left\{
    \alen{p}{q-k+2}{k-1} -\alen{p}{q-k+1}{k-1},
    \max\left( 0, \alen{p+1}{q-k}{k} -\alen{p}{q-k+1}{k-1} \right)
  \right\}\\
&+\min\left\{
    \alen{p}{q-k+1}{k} -\alen{p}{q-k}{k},
    \max\left( 0, \alen{p+1}{q-k}{k} -\alen{p}{q-k}{k} \right)
  \right\}\\
&+\min\left\{
    \alen{p}{q-k}{k+1} -\alen{p}{q-k-1}{k+1},
    \max\left( 0, \alen{p+1}{q-k}{k} -\alen{p}{q-k-1}{k+1} \right)
  \right\}\\
&\hspace{3.8pt}\vdots\\
&+\min\left\{
    \alen{p}{2}{q-1} -\alen{p}{1}{q-1},
    \max\left( 0, \alen{p+1}{q-k}{k} -\alen{p}{1}{q-1} \right)
  \right\}\\
&+\min\left\{
    \alen{p}{1}{q} -\alen{p}{0}{q},
    \max\left( 0, \alen{p+1}{q-k}{k} -\alen{p}{0}{q} \right)
  \right\}
\end{aligned}
\end{equation}
Every use of the $\alen{*}{*}{*}$ notation as the length of a row is justified by either Lemma~\ref{lem:d9e2p} or the Induction Hypothesis.

Let us first deal with the top $k-1$ lines.
All first items of the $\min\{\ ,\ \}$ terms are non-negative and Lemma~\ref{lem:dvlke} implies $\alen{p+1}{q-k}{k} \leq \alen{p}{q-k+1}{k-1}$ so that the $(k-1)$-th line is zero.
Adding Lemma~\ref{lem:s2eda} to this argument, we see that all the $\min\{\ ,\ \}$ terms in the initial $k-1$ lines are zero.

Next, since Lemma~\ref{lem:ls9an} implies $\alen{p+1}{q-k}{k} \geq \alen{p}{q-k}{k}$, the $k$-th line must be equal to
\begin{align*}
\min&\left\{
   \fwdn{p}{q-k+1}{k-1} -\fwdn{p}{q-k}{k},
   \alen{p+1}{q-k}{k} -\alen{p}{q-k}{k}
\right\}\\
&= \min\left\{
   \fwdn{p}{q-k+1}{k-1},
   \alen{p+1}{q-k}{k} -\alen{p}{q-k}{k} +\fwdn{p}{q-k}{k}
\right\}
-\fwdn{p}{q-k}{k}\\
&= \fwdn{p+1}{q-k}{k} -\fwdn{p}{q-k}{k}.
\end{align*}

Focusing on the $(k+1)$-th line, we can combining Lemma~\ref{lem:cr4s1} and Lemma~\ref{lem:ls9an} to claim
\begin{equation*}
\alen{p}{q-k}{k+1} \leq \alen{p}{q-k}{k} \leq \alen{p+1}{q-k}{k}
\end{equation*}
so that the $\min\{\ \,\ \}$ term must be equal to its left item.
Adding Lemma~\ref{lem:s2eda} to this argument, we can claim the same of all the $(k+1)$-th through $q$-th lines, and these are equal to the following:
\begin{equation}
\label{eq:j7u1o}
\begin{aligned}
\alen{p}{q-k}{k+1} -\alen{p}{q-k-1}{k+1}
&= \fwdn{p}{q-k}{k} -\fwdn{p}{q-k-1}{k+1}\\
&\hspace{5pt}\vdots\\
\alen{p}{2}{q-1} -\alen{p}{1}{q-1}
&= \fwdn{p}{2}{q-2} -\fwdn{p}{1}{q-1}\\
\alen{p}{1}{q} -\alen{p}{0}{q}
&= \fwdn{p}{1}{q-1}
\end{aligned}
\end{equation}

Finally, adding the $k$-th line term $\fwdn{p+1}{q-k}{k} -\fwdn{p}{q-k}{k}$ and all terms of~\eqref{eq:j7u1o} to the initial rigging of either $0$ or $-\fwdn{p+1}{q-k}{k-1}$, we arrive at the rigging of either $\fwdn{p+1}{q-k}{k}$ or $-\fwdn{p+1}{q-k}{k-1} +\fwdn{p+1}{q-k}{k}$ as claimed.
\end{proof}

Recall that an $\ftil_p$ action on a rigged configuration will add a box to its $p$-th part and change the riggings of the $p$-th and its neighboring parts, but not affect any other parts.
Hence, the preceding three lemmas describe everything there is to know about $\ftil_{p}^{\,\fwdfx{p}{q}} rc_-$.
It is now tedious but easy to confirm that we have reached the goal of our induction step.

This concludes our lengthy induction step and we have finally obtained a description of the rigged configuration $\bftil^{\,\fwdfxv} \RCinfv = rc_{1,n}$.
Let us specialize Definition~\ref{def:rc} to the case of~$rc_{1,n}$ and rewrite it in its simplified form for easy reference.
The new length symbol $\fwdlen{i}{k}$ appearing in the following definition is the special case $\alen{i}{n-i-k+2}{k}$ of the existing symbol.

\begin{definition}
\label{def:lenrig}
Let $\fwdfxv = (\fwdfx{i}{j})_{i,j}$ be a collection of non-negative integers that is weakly increasing with respect to the index~$i$ and let $\fwdnv = (\fwdn{i}{j}{k})_{i,j,k}$ be its extension obtained through Algorithm~\ref{alg:1}.
Define $rc_{\fwdfxv}$ to be the rigged configuration whose $i$-th part, for each $1\leq i \leq n$, is a rigged partition of height at most $n-i+1$ such that its $k$-th row is of length
\begin{align}
\label{eq:len}
\fwdlen{i}{k} &= \sum_{j=1}^{n-i-k+2} \fwdn{i}{j}{k-1} - \sum_{j=1}^{n-i-k+1} \fwdn{i}{j}{k},\\
\intertext{and has rigging}
\label{eq:rig}
\fwdrig{i}{k} &= -\fwdn{i}{n-i-k+2}{k-1} + \fwdn{i}{n-i-k+2}{k},
\end{align}
for $1 \leq k \leq n-i+1$.
The specified rigging $\fwdrig{i}{k}$ is to be ignored whenever the corresponding length $\fwdlen{i}{k}$ is zero.
\end{definition}

The result of the lengthy induction argument we gave in this section may be summarized as follows.
The second claim given here follows directly from the first claim and Lemma~\ref{lem:bij}.

\begin{theorem}\label{thm:fwdgenA}
Let $\fwdfxv = (\fwdfx{i}{j})_{i,j}$ be a collection of non-negative integers that is weakly increasing with respect to the index~$i$, and let $\fwdnv = (\fwdn{i}{j}{k})_{i,j,k}$ be its extension specified by Algorithm~\ref{alg:1}.
Then, $\bftil^{\,\fwdfxv} \RCinfv$ is equal to $rc_{\fwdfxv}$.
Each element of the crystal $\RCinf$ may be written uniquely in the form $rc_{\fwdfxv}$, for some weakly increasing $\fwdfxv\in\fwdfxvset$.
\end{theorem}

This achieves our first goal of listing all elements of the crystal $\RCinf \cong \Binf$ in a manner that does not involve Kashiwara operator applications.
Note that we may now freely use Claim~\ref{claim} as a valid result for any meaningful choice of the indices such that $j\neq 0$.
This also allows us to freely use all the lemmas in this section that depended on Claim~\ref{claim} being true.

The $i$-th rigged partition of the rigged configuration $rc_{\fwdfxv}$ may be presented pictorially as follows:
\begin{equation}\label{eq:fwdnui}
\raisebox{-2.5cm}{\includegraphics{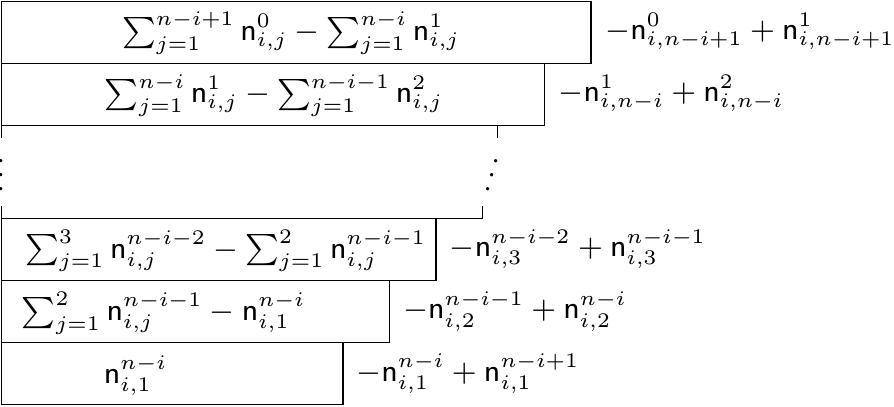}}
\end{equation}
The definition of $rc_{\fwdfxv}$ bounds the height of this $i$-th partition to~$n-i+1$.
However, for $i \leq \frac{n}{2}$, we know from Lemma~\ref{lem:fwdC} that the numbers within many of the lower rows of~\eqref{eq:fwdnui} are zero, and one can claim that the height of the $i$-th partition is at most $\min\{i,n-i+1\}$.
In fact, this symmetric bound on the height should be expected in view of the $A_n$-type Dynkin diagram automorphism $i\mapsto n-i+1$.
Although we could have removed some of the computations appearing in Algorithm~\ref{alg:1} by reflecting this bound, we have not done so in favor of a presentation that is more uniform over the partitions.

\section{Distinguishing the Elements of $\RCinf$}
\label{sec:4}

A method of listing each of the elements of $\RCinf$ precisely once was obtained in the previous section.
This method was essentially a bijective map from $\fwdfxvset$ to~$\RCinf$.
Although such a description of~$\RCinf$ is more explicit than its original presentation as a connected component, it still does not enable one to distinguish an element of~$\RCinf$ from an arbitrary rigged configuration.
The goal of this section is to resolve this situation through a map from~$\RCinf$ to~$\fwdfxvset$.

Note that, by definition, a row of length zero, i.e., a non-existent row, in a rigged partition does not carry a rigging.
Hence, the bijectivity of the map $\fwdfxv \mapsto rc_{\fwdfxv}$ must mean that, when $\fwdlen{i}{k}$ is zero, the value $\fwdrig{i}{k}$, as given by~\eqref{eq:rig}, is redundant information.
Given an element of~$\RCinf$ whose $\fwdlen{i}{k}$ value is zero, for some meaningful choice of indices, it should, at the worst, be possible to compute the corresponding fake rigging $\fwdrig{i}{k}$ from other information of the rigged configuration.
The next lemma shows that the situation is actually much simpler.

\begin{lemma}
\label{lem:l2r1}
Let $\fwdfxv = (\fwdfx{i}{j})_{i,j}$ be a collection of non-negative integers that is weakly increasing with respect to the index~$i$, and let $\fwdnv = (\fwdn{i}{j}{k})_{i,j,k}$ be its extension specified by Algorithm~\ref{alg:1}.
If a $\fwdlen{i}{k}$ value, computed according to~\eqref{eq:len}, is zero, then the corresponding $\fwdrig{i}{k}$ value, computed according to~\eqref{eq:rig}, is also zero, for every meaningful choice of the indices.
\end{lemma}
\begin{proof}
Since $\fwdlen{i}{k} = \alen{i}{n-i-k+2}{k}$ and $\fwdrig{i}{k} = -\fwdn{i}{n-i-k+2}{k-1} + \fwdn{i}{n-i-k+2}{k}$, this follows directly from Lemma~\ref{lem:f8s3k}.
\end{proof}

Given an arbitrary $rc_{\fwdfxv} \in\RCinf$, one can read off its $\fwdlen{i}{k}$ and $\fwdrig{i}{k}$ values, regardless of whether we know the corresponding $\fwdfxv$ value.
A certain $k$-th row of the $i$-th partition may be absent in $rc_{\fwdfxv}$, for a meaningful choice of the indices $i$ and~$k$, but the above lemma allows us to fill in even the corresponding $\fwdrig{i}{k}$ value.

\begin{algorithm}
 \For{$k=0,1,\dots,n-1$}{
   $j\leftarrow 1$\;
   $i\leftarrow n-j-k+1$\;
   $\fwdn{i}{j}{k} \leftarrow \fwdlen{i}{k+1}$\;
   \For{$j=2,3,\dots,n-k$}{
     $i \leftarrow n-j-k+1$\;
     $\fwdn{i}{j}{k} \leftarrow \fwdrig{i+1}{k+1} +\fwdn{i+1}{j-1}{k}
                      -\min\big\{0, \fwdlen{i+1}{k+1}-\fwdlen{i}{k+1}\big\}$\;
   }
 }
 \For{$t = 1, 2, ..., n-1$}{
   \For{$k = 0, 1, ..., n-t-1$}{
     $j \leftarrow 1$\;
     $i \leftarrow n-j-k+1-t$\;
     $\fwdn{i}{j}{k} \leftarrow \fwdlen{i}{k+1}
                      -\sum_{x=j+1}^{j+t} \fwdn{i}{x}{k}
                      +\sum_{x=j}^{j+t-1} \fwdn{i}{x}{k+1}$\;
     \For{$j = 2, 3, ..., n-k-t$}{
       $i \leftarrow n-j-k+1-t$\;
       $\fwdn{i}{j}{k} \leftarrow \fwdn{i+1}{j-1}{k+1}
               -\min\left\{ 0,
               \begin{aligned}
                 &\fwdlen{i+1}{k+1} -\fwdlen{i}{k+1}\\
                 &-\textstyle\sum_{x=j  }^{j+t-1} \fwdn{i+1}{x}{k} 
                  +\textstyle\sum_{x=j-1}^{j+t-2} \fwdn{i+1}{x}{k+1}\\
                 &+\textstyle\sum_{x=j+1}^{j+t  } \fwdn{i}{x}{k}
                  -\textstyle\sum_{x=j  }^{j+t-1} \fwdn{i}{x}{k+1}
               \end{aligned}
               \right\}$\;
     }
   }
 }
 \For{$j=1,2,\dots,n$}{
   \For{$i=1,2,\dots,n-j+1$}{
     $\fwdfx{i}{j} \leftarrow \fwdn{i}{j}{0}$\;
   }
 }
\caption{Recovery of $\fwdfxv = (\fwdfx{i}{j})_{i,j}$ from $\{\fwdlen{i}{k},\fwdrig{i}{k} \}_{i,k}$}\label{alg:2}
\end{algorithm}

The weakly increasing $\fwdfxv = (\fwdfx{i}{j})_{i,j}$ can now be recovered from the collection $\{ \fwdlen{i}{k}, \fwdrig{i}{k} \}_{i,k}$ through the process given by Algorithm~\ref{alg:2}, which is precisely the reversal of the composed $\fwdfxv \mapsto \fwdnv \mapsto \{\fwdlen{i}{k},\fwdrig{i}{k} \}_{i,k}$ mapping.
All the assignments given by Algorithm~\ref{alg:2} are equivalent to those made in Algorithm~\ref{alg:1}, \eqref{eq:len}, and~\eqref{eq:rig}, and the only difference is in the order of assignments, which has been carefully crafted so that all the terms appearing on the right-hand side of each assignment are available at that point.

Let us review the first two of the assignments as examples.
With the substitutions $j\leftarrow 1$ and $i \leftarrow n-j-k+1$, the very first assignment $\fwdn{i}{j}{k} \leftarrow \fwdlen{i}{k+1}$ of Algorithm~\ref{alg:2} may be written as
\begin{equation}
\fwdn{n-k}{1}{k} \leftarrow \fwdlen{n-k}{k+1},
\end{equation}
and this is precisely what we obtain when the substitutions $i \leftarrow n-k$ and $k \leftarrow k+1$ are made into~\eqref{eq:len}, with the only difference being that the left-hand and right-hand sides of the equality are interchanged.

The second assignment is
\begin{equation}
\fwdn{n-j-k+1}{j}{k}
 \leftarrow \fwdrig{n-j-k+2}{k+1} +\fwdn{n-j-k+2}{j-1}{k}
  -\min\left\{0,\fwdlen{n-j-k+2}{k+1}-\fwdlen{n-j-k+1}{k+1}\right\},
\end{equation}
with the appropriate substitutions.
Making the substitutions $i \leftarrow n-j-k+2$ and $k \leftarrow k+1$ into~\eqref{eq:rig} results in
\begin{equation}
\fwdrig{n-j-k+2}{k+1} = -\fwdn{n-j-k+2}{j-1}{k} + \fwdn{n-j-k+2}{j-1}{k+1},
\end{equation}
so that the above is equivalent to
\begin{equation}\label{eq:f23x}
\fwdn{n-j-k+1}{j}{k}
\leftarrow \fwdn{n-j-k+2}{j-1}{k+1}
 -\min\left\{0,\fwdlen{n-j-k+2}{k+1}-\fwdlen{n-j-k+1}{k+1}\right\}.
\end{equation}
Let us now work our way from the defining relation
\begin{equation}
\fwdn{i}{j}{k} = \min\left\{
  \fwdn{i-1}{j+1}{k-1},
  \begin{aligned}
    \textstyle
    \sum_{x=1}^{j  } \fwdn{i}{x}{k-1} 
   -\sum_{x=1}^{j-1} \fwdn{i}{x}{k}
   -\sum_{x=1}^{j  } \fwdn{i-1}{x}{k-1}
   +\sum_{x=1}^{j  } \fwdn{i-1}{x}{k}
  \end{aligned}
\right\}.
\end{equation}
With the substitutions $i \leftarrow n-j-k+2$, $j\leftarrow j-1$, and $k \leftarrow k+1$, followed by applications of~\eqref{eq:len}, this becomes
\begin{equation}
\fwdn{n-j-k+2}{j-1}{k+1}
= \min\left\{
  \fwdn{n-j-k+1}{j}{k},
    \fwdlen{n-j-k+2}{k+1}
   -\fwdlen{n-j-k+1}{k+1}
   +\fwdn{n-j-k+1}{j}{k}
  \right\},
\end{equation}
which may be written in the form
\begin{equation}
\fwdn{n-j-k+2}{j-1}{k+1}
= \fwdn{n-j-k+1}{j}{k}
 +\min\left\{ 0, \fwdlen{n-j-k+2}{k+1} -\fwdlen{n-j-k+1}{k+1} \right\},
\end{equation}
and this is equivalent to~\eqref{eq:f23x}.

The remaining two assignments may seem more complicated, but we have already shown all the techniques that are necessary in verifying that even these two are compatible with the $\fwdfxv \mapsto \{\fwdlen{i}{k},\fwdrig{i}{k} \}_{i,k}$ mapping.
Thus, a procedure for recovering $\fwdfxv$ from any given $rc_{\fwdfxv}\in\RCinf$ is now available to us, and the procedure does not require any applications of the Kashiwara operators.

The mapping $rc_{\fwdfxv} \mapsto \fwdfxv$ can actually be used to distinguish elements of~$\RCinf$ from other elements of the larger rigged configuration crystal.
Note that Algorithm~\ref{alg:2}, with the riggings set to zero for any of the missing rows, can be applied to any rigged configuration, regardless of whether it belongs to~$\RCinf$.
Suppose one is given an arbitrary rigged configuration.
If the $\fwdfxv$ computed by Algorithm~\ref{alg:2} for this rigged configuration belongs to~$\fwdfxvset$ and its image under the $\fwdfxv \mapsto rc_{\fwdfxv}$ mapping brings back the rigged configuration one started out with, then the element clearly belongs to~$\RCinf$.
On the other hand, since the mapping given by Algorithm~\ref{alg:2} is precisely the inverse of the $\fwdfxv \mapsto \{\fwdlen{i}{k},\fwdrig{i}{k} \}_{i,k}$ mapping for elements of $\RCinf$, every element of~$\RCinf$ will pass this test.

Note that verifying just the weakly increasing property of the~$\fwdfxv$ calculated from a given rigged configuration will not be sufficient, because we have no guarantee that the $\{\fwdlen{i}{k},\fwdrig{i}{k} \}_{i,k} \mapsto \fwdfxv$ mapping is injective on the larger crystal of all rigged configurations.
In fact, one can check that the value $\fwdrig{1}{1}$ is never used in Algorithm~\ref{alg:2}, so that the mapping is unlikely to be injective on the larger crystal.

\section{Isomorphism Between $\RCinf$ and $\Tinf$}

The description of~$\RCinf$ we gave in the previous sections was essentially a specialization of the bijection $\fwdfxv \mapsto \bftil^{\,\fwdfxv} \Binfv$ defined between $\fwdfxvset$ and~$\Binf$, given by Lemma~\ref{lem:bij}, into the bijection $\fwdfxv \mapsto rc_{\fwdfxv}$ defined between $\fwdfxvset$ and~$\RCinf$.
Let us explain that the marginally large tableau realization $\Tinf$ of~$\Binf$ also allows for the bijection $\fwdfxv \mapsto \bftil^{\,\fwdfxv} \Binfv$ to be expressed explicitly.

We acknowledge that the discussions of this section concerning $\Tinf$ are likely to have been known to the author of~\cite{Sai12}.
Although the paper does not involve $\Tinf$ explicitly, the arguments therein seem to be based on knowledge of what essentially amounts to the observations given below.

Given an $\fwdfxv = (\fwdfx{i}{j})_{i,j} \in\fwdfxvset$, let us define $\fwdtblm$ to be the unique marginally large tableau such that, for each $2\leq y \leq n+1$ and $1 \leq x \leq y-1$, the top $x$-many rows contain $\fwdfx{x}{n-y+2}$-many $y$-boxes.
For example, in the $A_3$-type case, we have
\begin{equation}\label{eq:fxT}
\fwdtblm = \
\raisebox{-1.45cm}{%
\includegraphics{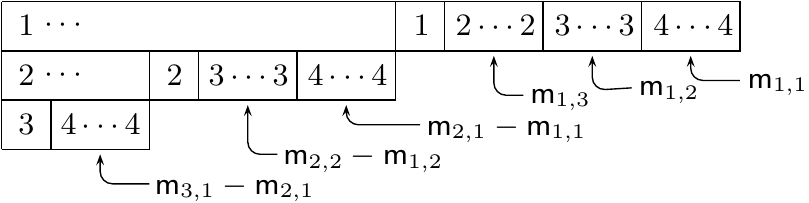}},
\end{equation}
where the number at the tail of each arrow is the count of blocks pointed to by the arrow.

\begin{proposition}
\label{prop:1}
Let $\fwdfxv = (\fwdfx{i}{j})_{i,j} \in\fwdfxvset$ be a collection of non-negative integers that is weakly increasing with respect to the index~$i$, and let $\Tinfv\in \Tinf$ be the highest weight element.
Then, the tableau $\bftil^{\,\fwdfxv} \Tinfv \in \Tinf$ is~$\fwdtblm$.
\end{proposition}

This is a straightforward computational result that requires only the knowledge of the Kashiwara operator actions on the marginally large tableaux and no other special technique.
Given an $\fwdfxv\in\fwdfxvset$, the combined special structures of $\fwdfxv$ and $\bftil^{\,\fwdfxv}$ make the computation very easy.
For example, in the $A_3$-type crystal $\Tinf$, one can immediately realize that
\begin{equation}
\bftil^{\,\fwdfx{*}{1}} \Tinfv = \
\raisebox{-1.33cm}{%
\includegraphics{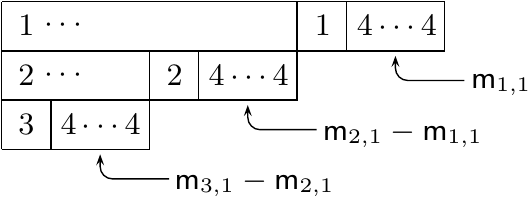}}.
\end{equation}
Further applications of $\bftil^{\,\fwdfx{*}{2}}$ and $\bftil^{\,\fwdfx{*}{3}}$ cannot disturb the $4$-blocks, and one can conclude that the end result must be~\eqref{eq:fxT}.

One can see, either through Lemma~\ref{lem:bij} or simply by reviewing the definition of~$\fwdtblm$, that every element of the crystal~$\Tinf$ may be expressed in the form~$\fwdtblm$ for some uniquely determined $\fwdfxv\in\fwdfxvset$.
Note that the bijection $\fwdfxv \mapsto \fwdtblm$ between $\fwdfxvset$ and~$\Tinf$ is so direct that neither direction of the bijection requires any meaningful amount of computation.

The following claim is now a consequence of having labeled the elements of both~$\RCinf$ and~$\Tinf$ through Lemma~\ref{lem:bij}.

\begin{theorem}
\label{thm:fwd}
The map that sends the tableau $\fwdtblm$ to the rigged configuration $rc_{\fwdfxv}$, for each $\fwdfxv\in\fwdfxvset$, is an isomorphism between $\Tinf$ and $\RCinf$.
\end{theorem}

Explicit computations of the map $\fwdtblm \mapsto rc_{\fwdfxv}$ are made possible by Algorithm~\ref{alg:1}.
Furthermore, the inverse map $rc_{\fwdfxv} \mapsto \fwdtblm$ may be computed explicitly, even when the input element of~$\RCinf$ is presented without its corresponding~$\fwdfxv$ value, using Algorithm~\ref{alg:2}.

\section{Crystals $\Bla$ as a Subset of $\RCinf$}

In this section, we describe the irreducible highest weight crystal $\Bla$ through the description of $\RCinf$ obtained in the previous section.

Given $\la\in P^+$, we define $\fwdfxvset^{\la}$ to be the set of $\fwdfxv = (\fwdfx{i}{j})_{i,j} \in\fwdfxvset$ satisfying the condition
\begin{align}
\sum_{x=1}^{n-j+1} (\fwdfx{i}{x} - \fwdfx{i-1}{x})
\leq \lambda(h_i) + \sum_{x=1}^{n-j} (\fwdfx{i+1}{x} - \fwdfx{i}{x}),
\end{align}
for each pair of indices $1\leq i\leq j\leq n$, where all $\fwdfx{0}{x}$ are interpreted as zeros.
Recall from Lemma~\ref{lem:cryimb} that a realization of~$\Bla$ may be obtained by describing the image of a certain strict crystal embedding.
The following is a restatement of a result from~\cite{Lee14} in the language used by this paper.

\begin{lemma}
For each $\la \in P^+$, the image of the strict crystal embedding
$\Bla \hookrightarrow \Tinf \otimes \Tvla$
that maps $b_\lambda$ to $\Tinfv \otimes \tvla$ is
$\left\{ \fwdtblm \otimes \tvla \,\left|\, \fwdfxv \in \fwdfxvset^{\la} \right.\right\}$.
\end{lemma}

Since we know from Theorem~\ref{thm:fwd} that the mapping $\fwdtblm \mapsto rc_{\fwdfxv}$ is a crystal isomorphism between $\Tinf$ and $\RCinf$ that maps $\Tinfv$ to~$\RCinfv$, the following is now evident.

\begin{theorem}
\label{thm:RCla}
For each $\la \in P^+$, the image of the strict crystal embedding
$\Bla \hookrightarrow \RCinf \otimes \Tvla$
that maps $b_\lambda$ to $\RCinfv \otimes \tvla$ is
$\left\{ rc_{\fwdfxv} \otimes \tvla \,\left|\, \fwdfxv \in \fwdfxvset^{\la} \right.\right\}$.
\end{theorem}

Thus, we have a new realization of the irreducible highest weight crystal~$\Bla$.

\section{Reverse Tableau Approach}

The developments of the previous sections were naturally connected to the structure of~$\Tinf$, the marginally large tableau realization of~$\Binf$.
In this section, we will present analogues that are connected to~$\Rinf$, the marginally large \emph{reverse} tableau realization of~$\Binf$.
We will be very brief and present just the main facts.

Let us consider a collection of non-negative integers $\rvsfxv = (\rvsfx{i}{j})_{i,j}$, where the indices span over the range $1 \leq j \leq n$ and $j \leq i \leq n$.
We will say that such an $\rvsfxv$ is \emph{weakly decreasing} with respect to the index~$i$, if it satisfies
\begin{equation*}
\rvsfx{i}{j} \geq \rvsfx{i+1}{j},
\end{equation*}
for every meaningful choice of the indices~$i$ and~$j$, and the set of all weakly decreasing~$\rvsfxv$ will be denoted by~$\rvsfxvset$.

To each collection $\rvsfxv$ of non-negative integers, we associate the product of lowering Kashiwara operators
$\bftil^{\,\rvsfxv} = 
  \bftil^{\,\rvsfx{*}{n}}
  \cdots
  \bftil^{\,\rvsfx{*}{2}}
  \bftil^{\,\rvsfx{*}{1}}$,
where each
$\bftil^{\,\rvsfx{*}{j}} =
  \ftil_{j}^{\,\rvsfx{j}{j}}
  \ftil_{j+1}^{\,\rvsfx{j+1}{j}}
  \cdots
  \ftil_{n}^{\,\rvsfx{n}{j}}$.

Given an $\rvsfxv = (\rvsfx{i}{j})_{i,j} \in\rvsfxvset$, we define $\rvstblm$ to be the marginally large \emph{reverse} tableau that is a natural generalization to the $A_n$-type of the following reverse tableau drawn for the $A_3$-type.
\begin{equation}
\label{eq:mlrt}
\rvstblm = \
\raisebox{-1.4cm}{\includegraphics{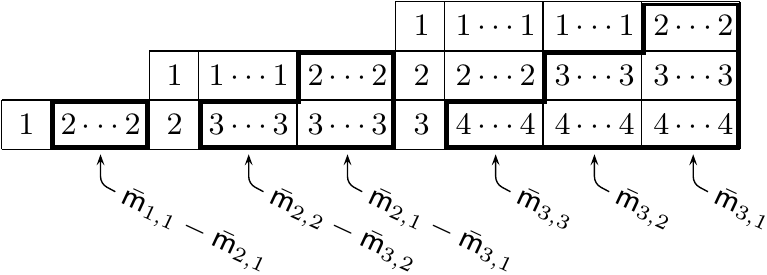}}
\end{equation}
A formal description of $\rvstblm$ can be given as follows.
Consider the general form of a marginally large reverse tableau together with the highest weight reverse tableau of the same shape (which is usually not marginally large).
When boxes whose entries agree on the two reverse tableaux are removed from the marginally large tableau, we are left with $n$ separate reverse tableaux.
In the above $A_3$-type reverse tableau, the three leftover parts are marked with thick lines.
The $\rvstblm$ is such that, for each $1\leq x \leq n$ and $1 \leq y \leq x$, the $y$-th row from the bottom of the $x$-th part from the left consists of $\big\{\sum_{j=1}^{x-y+1} (\rvsfx{x}{j}-\rvsfx{x+1}{j})\big\}$-many $(x-y+2)$-boxes, where we are setting $\rvsfx{n+1}{j} = 0$.

\begin{proposition}
\label{prop:2}
Let $\rvsfxv = (\rvsfx{i}{j})_{i,j} \in\rvsfxvset$ be a collection of non-negative integers that is weakly decreasing with respect to the index~$i$, and let $\Rinfv\in \Rinf$ be the highest weight element.
Then, $\bftil^{\,\rvsfxv} \Rinfv \in \Rinf$ is the marginally large reverse tableau $\rvstblm$.
\end{proposition}

As was with Proposition~\ref{prop:1}, this is a straightforward computational result that requires only the knowledge of the Kashiwara operator actions on the marginally large reverse tableaux.

It is clear from the definition of $\rvstblm$ that any marginally large reverse tableau $\rvstbl \in\Rinf$ may be written in the form $\rvstbl = \rvstblm$, for some $\rvsfxv\in\rvsfxvset$ that is uniquely determined by~$\rvstbl$.
Since the set of marginally large reverse tableaux represents each and every element of $\Binf\cong\Rinf$ exactly once, Proposition~\ref{prop:2} implies the following.

\begin{lemma}
\label{lem:bijR}
Let $\Binfv$ be the highest weight element of~$\Binf$.
The function that maps~$\rvsfxv$ to~$\bftil^{\,\rvsfxv} \Binfv$ is a bijection from~$\rvsfxvset$ to~$\Binf$.
\end{lemma}
This is the analogue of Lemma~\ref{lem:bij} that corresponds to the reduced expression
\begin{equation}
(s_n)(s_{n-1} s_n) \cdots\cdots (s_2 s_3\cdots s_n)(s_1 s_2 \cdots s_n)
\end{equation}
for the longest Weyl group element, just as Lemma~\ref{lem:bij} relates to~\eqref{eq:redexp}.
We acknowledge that it should be possible to derive this result from Lemma~\ref{lem:bij} through certain computations developed in~\cite{BerZel93,BerZel96}.
However, we found it easier to obtain this directly from the marginally large reverse tableau realization of~$\Binf$.

\begin{algorithm}
 \For{$j=1,2,\dots,n$}{
   \For{$i=j,j+1,\dots,n$}{
     $\rvsn{i}{j}{0} \leftarrow \rvsfx{i}{j}$\;
   }
 } 
 \For{$k=1,2,\dots,n$}{
   \For{$j = 1,2,\dots,n-k+1$}{
     $\rvsn{n}{j}{k} \leftarrow 0$\;
     \For{$i = n-1,n-2,\dots,j+k-1$}{
       $\rvsn{i}{j}{k} \leftarrow \min\left\{
          \rvsn{i+1}{j+1}{k-1},
          \begin{aligned}
            \textstyle\sum_{x=1}^{j  } \rvsn{i}{x}{k-1} 
           -\textstyle\sum_{x=1}^{j  } \rvsn{i+1}{x}{k-1}
           -\textstyle\sum_{x=1}^{j-1} \rvsn{i}{x}{k}
           +\textstyle\sum_{x=1}^{j  } \rvsn{i+1}{x}{k}
          \end{aligned}
       \right\}$\;
     }
   }
 }
\caption{Generation of $\rvsnv = (\rvsn{i}{j}{k})_{i,j,k}$ from $\rvsfxv = (\rvsfx{i}{j})_{i,j}$}\label{alg:3}
\end{algorithm}

The extension $\rvsnv = (\rvsn{i}{j}{k})_{i,j,k}$ of a given $\rvsfxv = (\rvsfx{i}{j})_{i,j}\in\rvsfxvset$ is defined through Algorithm~\ref{alg:3}.
Given an $\rvsfxv\in\rvsfxvset$, we define $rc_{\rvsfxv}$ to be the rigged configuration whose $k$-th row (from the top) of its $i$-th rigged partition is of length
\begin{align}
\label{eq:len2}
\rvslen{i}{k} &= \sum_{j=1}^{i-k+1} \rvsn{i}{j}{k-1} - \sum_{j=1}^{i-k} \rvsn{i}{j}{k}\\
\intertext{and has rigging}
\label{eq:rig2}
\rvsrig{i}{k} &= -\rvsn{i}{i-k+1}{k-1} + \rvsn{i}{i-k+1}{k},
\end{align}
for each $1\leq i \leq n$ and $1 \leq k \leq i$.
The height of the $i$-th partition is bounded by $\min\{i,n-i+1\}$.
The $A_4$-type rigged configuration $rc_{\rvsfxv}$ is provided below as an example.
\begin{center}
\includegraphics{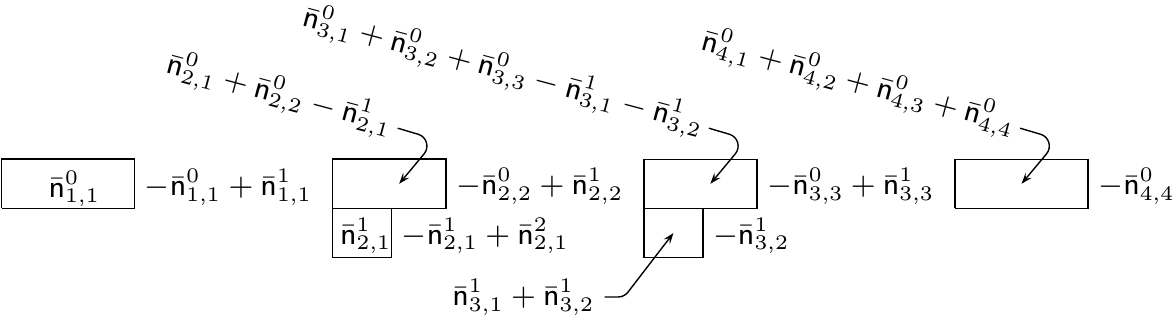}
\end{center}
\begin{align*}
\rvsn{3}{1}{1} &= \min\big\{
   \rvsn{4}{2}{0},
   (\rvsn{3}{1}{0})
  -(\rvsn{4}{1}{0})
\big\}\\
\rvsn{2}{1}{1} &= \min\big\{
   \rvsn{3}{2}{0},
   (\rvsn{2}{1}{0})
  -(\rvsn{3}{1}{0}-\rvsn{3}{1}{1})
\big\}\\
\rvsn{1}{1}{1} &= \min\big\{
   \rvsn{2}{2}{0},
   (\rvsn{1}{1}{0})
  -(\rvsn{2}{1}{0}-\rvsn{2}{1}{1})
\big\}\\
\rvsn{3}{2}{1} &= \min\big\{
   \rvsn{4}{3}{0}, 
   (\rvsn{3}{1}{0}+\rvsn{3}{2}{0}-\rvsn{3}{1}{1})
  -(\rvsn{4}{1}{0}+\rvsn{4}{2}{0})
\big\}\\
\rvsn{2}{2}{1} &= \min\big\{
   \rvsn{3}{3}{0},
   (\rvsn{2}{1}{0}+\rvsn{2}{2}{0}-\rvsn{2}{1}{1})
  -(\rvsn{3}{1}{0}+\rvsn{3}{2}{0}-\rvsn{3}{1}{1}-\rvsn{3}{2}{1})
\big\}\\
\rvsn{3}{3}{1} &= \min\big\{
   \rvsn{4}{4}{0},
   (\rvsn{3}{1}{0}+\rvsn{3}{2}{0}+\rvsn{3}{3}{0}-\rvsn{3}{1}{1}-\rvsn{3}{2}{1})
  -(\rvsn{4}{1}{0}+\rvsn{4}{2}{0}+\rvsn{4}{3}{0})
\big\}\\
\rvsn{2}{1}{2} &= \min\big\{
   \rvsn{3}{2}{1},
   (\rvsn{2}{1}{1})
  -(\rvsn{3}{1}{1})
\big\}
\end{align*}

As with Theorem~\ref{thm:fwdgenA}, the following is mostly a computational result, with Lemma~\ref{lem:bijR} taking the role of Lemma~\ref{lem:bij}.

\begin{theorem}
\label{thm:4}
Let $\rvsfxv = (\rvsfx{i}{j})_{i,j}$ be a collection of non-negative integers that is weakly decreasing with respect to the index~$i$, and let $\rvsnv = (\rvsn{i}{j}{k})_{i,j,k}$ be its extension specified by Algorithm~\ref{alg:3}.
Then, $\bftil^{\,\rvsfxv} \RCinfv$ is equal to $rc_{\rvsfxv}$.
Each element of $\RCinf$ may be written uniquely in the form $rc_{\rvsfxv}$, for some weakly decreasing $\rvsfxv\in\rvsfxvset$.
\end{theorem}

The following claim allows one to set any rigging, that is missing due to the corresponding row being non-existent, to zero.

\begin{lemma}
Let $\rvsfxv = (\rvsfx{i}{j})_{i,j}$ be a collection of non-negative integers that is weakly decreasing with respect to the index~$i$, and let $\rvsnv = (\rvsn{i}{j}{k})_{i,j,k}$ be its extension specified by Algorithm~\ref{alg:3}.
If a $\rvslen{i}{k}$ value, computed according to~\eqref{eq:len2}, is zero, then the corresponding $\rvsrig{i}{k}$ value, computed according to~\eqref{eq:rig2}, is also zero, for every meaningful choice of the indices.
\end{lemma}

\begin{algorithm}
 \For{$k=0,1,\dots,n-1$}{
   $j \leftarrow 1$\;
   $i \leftarrow j+k$\;
   $\rvsn{i}{j}{k} \leftarrow \rvslen{i}{k+1}$\;
   \For{$j=2,3,\dots,n-k$}{
     $i \leftarrow j+k$\;
     $\rvsn{i}{j}{k} \leftarrow \rvsrig{i-1}{k+1} +\rvsn{i-1}{j-1}{k}
                      -\min\big\{0, \rvslen{i-1}{k+1}-\rvslen{i}{k+1}\big\}$\;
   }
 }
 \For{$t = 1, 2, ..., n-1$}{
   \For{$k = 0, 1, ..., n-t-1$}{
     $j \leftarrow 1$\;
     $i \leftarrow j+k+t$\;
     $\rvsn{i}{j}{k} \leftarrow \rvslen{i}{k+1}
                      -\sum_{x=j+1}^{j+t} \rvsn{i}{x}{k}
                      +\sum_{x=j}^{j+t-1} \rvsn{i}{x}{k+1}$\;
     \For{$j = 2, 3, ..., n-k-t$}{
       $i \leftarrow j+k+t$\;
       $\rvsn{i}{j}{k} \leftarrow \rvsn{i-1}{j-1}{k+1}
               -\min\left\{ 0,
               \begin{aligned}
                 &\rvslen{i-1}{k+1} -\rvslen{i}{k+1}\\
                 &-\textstyle\sum_{x=j  }^{j+t-1} \rvsn{i-1}{x}{k} 
                  +\textstyle\sum_{x=j-1}^{j+t-2} \rvsn{i-1}{x}{k+1}\\
                 &+\textstyle\sum_{x=j+1}^{j+t  } \rvsn{i}{x}{k}
                  -\textstyle\sum_{x=j  }^{j+t-1} \rvsn{i}{x}{k+1}
               \end{aligned}
               \right\}$\;
     }
   }
 }
 \For{$j=1,2,\dots,n$}{
   \For{$i=j,j+1,\dots,n$}{
     $\rvsfx{i}{j} \leftarrow \rvsn{i}{j}{0}$\;
   }
 }
\caption{Recovery of $\rvsfxv = (\rvsfx{i}{j})_{i,j}$ from $\{\rvslen{i}{k},\rvsrig{i}{k} \}_{i,k}$}\label{alg:4}
\end{algorithm}

The weakly decreasing $\rvsfxv$ can now be recovered from the $\{\rvslen{i}{k}, \rvsrig{i}{k}\}_{i,k}$ values through Algorithm~\ref{alg:4}.
Just as was explained in Section~\ref{sec:4}, we can use the map $\{\rvslen{i}{k}, \rvsrig{i}{k}\}_{i,k} \mapsto \rvsfxv$, given by Algorithm~\ref{alg:4}, to distinguish elements of~$\RCinf$ from other elements of the larger rigged configuration crystal.

The following can be obtained by combining Proposition~\ref{prop:2} and Theorem~\ref{thm:4}.

\begin{theorem}
\label{thm:rvs}
The map that sends the reverse tableau $\rvstblm$ to the rigged configuration $rc_{\rvsfxv}$, for each $\rvsfxv\in\rvsfxvset$, is an isomorphism between $\Rinf$ and~$\RCinf$.
\end{theorem}

As a byproduct of Theorem~\ref{thm:fwd} and Theorem~\ref{thm:rvs}, we obtain a crystal isomorphism between $\Tinf$ and~$\Rinf$, given by $\fwdtblm \mapsto rc_{\fwdfxv} = rc_{\rvsfxv} \mapsto \rvstblm$.
However, the authors are aware that both the Schensted bumping algorithm~\cite{Knu70,Sch61} and the Sch\"utzenberger sliding algorithm~\cite{Sch77} can be utilized to create isomorphisms between~$\Tinf$ and~$\Rinf$ that are more direct.
We wish to expand on this subject in a future work.

The reverse tableau analogue of Theorem~\ref{thm:RCla} will be provided next.
Given $\la\in P^+$, we define $\rvsfxvset^{\la}$ to be the set of $\rvsfxv = (\rvsfx{i}{j})_{i,j} \in\rvsfxvset$ satisfying the condition
\begin{equation}
\label{eq:d8d3s}
\sum_{x=1}^{j} (\rvsfx{i}{x} - \rvsfx{i+1}{x})
\leq \lambda(h_i) + \sum_{x=1}^{j-1} (\rvsfx{i-1}{x} - \rvsfx{i}{x}),
\end{equation}
for each pair of indices $1\leq j\leq i\leq n$, where any $\rvsfx{n+1}{x}$ are interpreted as zeros.

\begin{lemma}
For each $\la \in P^+$, the image of the strict crystal embedding
$\Bla \hookrightarrow \Rinf \otimes \Tvla$
that maps $b_\lambda$ to $\Rinfv \otimes \tvla$ is
$\left\{ \rvstblm \otimes \tvla \,\left|\, \rvsfxv \in \rvsfxvset^{\la} \right.\right\}$.
\end{lemma}

Combining this with Theorem~\ref{thm:rvs}, we obtain another description of~$\Bla$ given in terms of rigged configurations.

\begin{theorem}
For each $\la \in P^+$, the image of the strict crystal embedding
$\Bla \hookrightarrow \RCinf \otimes \Tvla$
that maps $b_\lambda$ to $\RCinfv \otimes \tvla$ is
$\left\{ rc_{\rvsfxv} \otimes \tvla \,\left|\, \rvsfxv \in \rvsfxvset^{\la} \right.\right\}$.
\end{theorem}

The image sets described by Theorem~\ref{thm:RCla} and the above claim must be the same set.
In other words, we have described the same set in two different ways.

\section*{Acknowledgments}

The authors are grateful to Travis Scrimshaw for valuable discussions and to Satoshi Naito for providing important information related to this work.
Some of this work was done while HL was visiting Department of Mathematics at UC Davis and most of the progress was made while the two authors were visiting Korea Institute for Advanced Study.
The authors thank them for their hospitality and support.
The Sage Mathematical Software~\cite{Sag08,Ste15} was used at an early stage of this work to experiment with examples.



\appendix

\vspace{3em}

\section*{Addendum}

The second version (\href{https://arxiv.org/abs/1604.04357v2}{v2}, May 2016) of this e-print (arXiv:1604.04357) had an Appendix section that was not present in its previous version (\href{https://arxiv.org/abs/1604.04357v1}{v1}, Apr 2016).
The Appendix has been removed because many errors and logical gaps were found in it.
Note that the name Roger Tian was added to and removed from the author list of this e-print together with the said Appendix.

In Nov 2016, R.Tian posted an updated version\footnote{The state of arXiv:1611.07869v2 is still far from suitable for publication.} of the Appendix as arXiv:1611.07869.
This was done without any consultation with H.Lee or J.Hong, even though H.Lee and J.Hong had been contacting the non-responsive R.Tian to discuss the future of the troubled contents.
Furthermore, the posted report illegitimately lists just R.Tian as its author.
The Appendix reported on a separate incomplete project that H.Lee had allowed R.Tian to join in after R.Tian implored H.Lee for a problem to work on.%
\footnote{The Appendix was essentially the second half of ``Ruoguang (Roger) Tian, \textit{Top to Random Shuffles and Characterization of Rigged Configurations of~$\mathfrak{B}(\infty)$ in Type~$A$}, UC~Davis Ph.D. Dissertation, June 2016,'' as it was less than a week prior to its final due date.}
J.Hong also contributed significantly to arXiv:1611.07869 throughout the May--Jul 2016 period while trying to fix the troubled contents.
We plan to provide more information concerning this matter in the future through another channel.\footnote{Probably \url{http://www.math.snu.ac.kr/~jinhong/arXiv.1611.07869/}.}

\end{document}